\documentclass[11pt]{article} 
\usepackage[a4paper, top=3.5cm, left=3.5cm, right=3.5cm, bottom=4cm]{geometry} 
\usepackage{amsmath, amssymb} 
\usepackage{enumerate} 
\usepackage[backend=biber, style=alphabetic, giveninits=true]{biblatex} 
\usepackage{mathtools} 
\usepackage{color} 
\usepackage{amsthm} 
\usepackage{asymptote} 
\usepackage[pdfpagelabels,colorlinks=true,citecolor=blue]{hyperref} 
\usepackage[labelfont={sc}]{caption} 
\usepackage{tikz} 
\usepackage[capitalize]{cleveref} 
\usetikzlibrary{cd} 
\usepackage{todonotes} 
 
\AtEveryBibitem{\clearfield{month}} 
\AtEveryBibitem{\clearfield{day}} 
\renewbibmacro{in:}{}

\long\def\ignore#1\recognize{} 
\let\svthefootnote\thefootnote 

\def\N{\mathbb{N}} 
\def\R{\mathcal{R}}

\def\Z{\mathbb{Z}} 
\def\Q{\mathbb{Q}}

\def\6{\partial} 
\def\66{{\partial}} 
\def\min{{\rm min}} 
 
\def\ord{{\rm ord}}

\def\Ker{{\rm Ker}}

\DeclareMathOperator{\GL}{GL}

\DeclareMathOperator{\odd}{odd} 
\DeclareMathOperator{\eve}{even} 
\DeclareMathOperator{\ch}{char} 

\def\ds{\displaystyle}

\def\F{{\mathbb{F}}}

\newcommand{\Rp}{\mathcal{R}_p} 
\newcommand{\Rpk}{\mathcal{R}_p^{(k)}} 
 
\newcommand{\Spk}{\mathcal{S}_p^{(k)}} 
\newcommand{\Spi}{\mathcal{S}_p^{(i)}} 
\newcommand{\Sp}{\mathcal{S}_p} 
\newcommand{\Cp}{\mathcal{C}_p} 
\newcommand{\expp}{\exp_p} 
\newcommand{\exppt}{\widetilde{\exp}_p} 
 
\renewcommand{\epsilon}{\varepsilon} 
\newcommand{\monominduced}{xeric } 
 
\def \osb{[\![} 
\def \csb{]\!]} 
 
\def \orb{(\!(} 
\def \crb{)\!)} 
 
\theoremstyle{definition} 
\newtheorem{defi}{Definition}[section] 
\newtheorem{ex}[defi]{Example} 
 
\theoremstyle{plain} 
\newtheorem{thm}[defi]{Theorem} 
\newtheorem{lem}[defi]{Lemma} 
\newtheorem{cor}[defi]{Corollary} 
\newtheorem{prop}[defi]{Proposition}

\newtheorem{prob}[defi]{Problem} 
 
\theoremstyle{remark} 
\newtheorem{rem}[defi]{Remark} 
 
\begin{document} 
\title{On Abel's Problem about Logarithmic Integrals in Positive Characteristic} 
\author{Florian F\"urnsinn, Herwig Hauser and Hiraku Kawanoue} 
\maketitle

\begin{abstract} 
Linear differential equations with polynomial coefficients over a field $K$ of positive characteristic $p$ and with local exponents in the prime field have a basis of solutions in the differential extension $\Rp=K(z_1, z_2, \ldots)\orb x\crb$ of $K(x)$, where $x'=1, z_1'=1/x$ and $z_i'=z_{i-1}'/z_{i-1}$. For differential equations of order $1$ it is shown that there exists a solution $y$ whose projections $y\vert_{z_{i+1}=z_{i+2}=\cdots=0}$ are algebraic over the field of rational functions $K(x, z_1, \ldots, z_{i})$ for all $i$. This can be seen as a characteristic $p$ analogue of Abel's problem about the algebraicity of logarithmic integrals. Further, the existence of infinite product representations of these solutions is shown. As a main tool $p^i$-curvatures are introduced, generalizing the notion of $p$-curvature. 
\let\thefootnote\relax\footnote{  The first and second authors were supported by the Austrian Science Fund FWF P-34765, which also enabled the organisation of a workshop in Lisbon in February 2023, where the three authors started their discussion. The third author was supported by JSPS Grant-in-Aid for Scientific Research (C), No. 20K03546. We thank A. Bostan for posing \cref{quest:algebraicity} for the exponential differential equation and X. Caruso, D. van Straten and S. Yurkevich  for fruitful discussions and for very valuable insights. Further, we thank S. Schneider for reading a preliminary version and helping with the computations.} 
\addtocounter{footnote}{-1}\let\thefootnote\svthefootnote 
\end{abstract}

\section{Introduction} \label{sec:intro} 
Niels Abel asked for criteria when a differential equation of the form 
\begin{equation} \label{eq:deq1} 
\frac{y'}{y}=a 
\end{equation} 
has, for $a$ a complex polynomial or a rational (respectively, algebraic) function, an algebraic solution $y$ (see Boulanger \cite[p.~93]{Bou97}). A necessary condition is that $a$ has only simple poles, as  is the case for $\frac{y'}{y}$, for any holomorphic or meromorphic $y$. For instance, if $b$ is a rational function and $k\in \Z\setminus \{0\}$, then $a \coloneqq \frac{1}{k}\frac{b'}{b}$ yields the algebraic solution $y=\sqrt[k]{b}$. Algorithmically, the problem has been solved by Risch \cite{Ris70}. In the present note, we address a similar problem for first order differential equations defined over a field $K$ of positive characteristic $p$. A first distinction is the fact that equations like \eqref{eq:deq1} need not have formal power series solutions, or, more generally, solutions of the form $x^\rho f$ for some power series $f\in K\osb x\csb$ and some $\rho\in \overline K$. For instance,  the exponential function $\exp\in \Q\osb x\csb $, solution of $y'=y$, cannot be reduced modulo $p$ to obtain a solution in $\F_p\osb x\csb$,  for any prime $p$, as all prime numbers appear in the denominators of the coefficients. And, indeed, making an unknown ansatz $y=\sum_i c_ix^i$ for the solution in characteristic $p$, and solving for the coefficients $c_i\in K$ recursively, a contradiction occurs once $i$ reaches $p$.\\ 
 
In \cite{FH23} F\"urnsinn and Hauser introduce for $K=\F_p$ the differential extension \[\Rp=\F_p(z_1, z_2, \ldots )\orb x\crb\] of the field of formal Laurent series $\F_p\orb x\crb$ by adjoining countably many variables $z_i$, equipped with the $\F_p$-derivation $\partial$ given by 
\begin{gather*} 
\partial x=1,\qquad \partial z_1 = \frac{1}{x},\\ 
\partial z_i=\frac{\partial z_{i-1}}{z_{i-1}}=\frac{1}{x\cdot z_1\cdots z_{i-1}}, \quad \text{for } i\geq 2. 
\end{gather*} 
As usual, we write $y'$ for $\partial y$ for $y\in \Rp$. The field of constants is $\Cp\coloneqq \R_p^p=\F_p(z_1^p, z_2^p, \ldots )\orb x^p\crb$. The action of $\partial$ decreases the order in $x$ of elements of $\Rp$ of the form $f(z)x^k$ by $1$, i.e., $(f(z)x^k)'\in x^{k-1}\F_p(z_1,z_2,\ldots).$ 
 
It is proven in \cite{FH23} that any differential equation $Ly=0$, for $L\in \F_p\osb x\csb[\partial]$ an operator of order $n$ with regular singularity at $0$ and local exponents in $\F_p$, has $n$ solutions $y_1,...,y_n$ in $\Rp$, linearly independent over $\Cp$, and even in $\F_p[z_1, z_2,\ldots]\osb x\csb$, the ring of power series in $x$ with polynomial coefficients in the $z_i$. 
 
For solutions $y$ in $\Rp$, the notion of algebraicity is more subtle. In general, $y$ depends on infinitely many variables $z_i$ and will almost never be algebraic over $\F_p(x, z_1, z_2,\ldots)$. But, for a given series $y(x,z)\in \F_p[z_1, z_2,\ldots]\osb x\csb$, one may require that the {\it projections} to series in finitely many variables, 
\[y(x,z_1,...,z_i,0,...)=y\vert_{z_{i+1}=z_{i+2}=\cdots=0},\] 
obtained by setting almost all $z$-variables equal to $0$, are algebraic over the field $\F_p(x, z_1, z_2,\ldots,z_i)$. This concept of algebraicity turns out to be very fruitful. The question then is whether $y$ can be approximated by such algebraic series, each involving an increasing number of $z$-variables. 
 
\begin{prob}\label{quest:algebraicity} 
Let $Ly=0$ be a differential equation with regular singularity at $0$ and polynomial or algebraic power series coefficients. Does there exist a $\Cp$-basis of solutions $y_1,\ldots,y_n\in \F_p[z_1,z_2,\ldots]\osb x\csb$ whose projections 
$y_j\vert_{z_{i+1}=z_{i+2}=\cdots=0}$ 
are algebraic over $\F_p(x, z_1,\ldots, z_{i})$? 
\end{prob} 
 
In particular, one may ask whether the {\em initial series} $y_j\vert_{z_1=z_2=\cdots=0}\in \F_p\osb x\csb$ of the basis are algebraic?\\ 
 
For first order  equations the answer is affirmative (see Theorem~\ref{thm:prod}): 
 
\begin{thm}\label{thm:introalg} 
Let $y'=ay$ be an equation of order $1$, regular at $0$, with local exponent $\rho=0$, where $a$ is some algebraic series in $\F_p\orb x\crb$. There exists a non-zero solution $y\in \F_p[z_1, z_2,\ldots]\osb x\csb$ with algebraic projections $y\vert_{z_{i+1}=\cdots=0}$ for all $i\geq 1$.  \end{thm} 
 
As solutions may be multiplied by arbitrary constants in $\Cp$, not all solutions will have this property. But for the chosen one even more is true. 
 
\begin{thm}\label{thm:introalgtwo} In the preceding situation, the specified solution $y$ can be written as an infinite product \[y=\prod_{i=0}^\infty h_i,\] where the factors $h_i$ belong to $1+x^{p^i}z_i\F_p[z_1, z_2,\ldots, z_i]\osb x\csb$ and are algebraic over $\F_p(x, z_1,\ldots, z_i)$. \end{thm} 
 
In the case of the exponential function, solution of $y'=y$, the factors $h_i$ can be described explicitly (Theorem \ref{thm:expt}); they are given by substituting in the polynomial \[H(t) = \prod_{k=1}^{p-1}\left(1-\frac{t}{k}\right)^k\] the variable $t$ by $(-1)^i g_i$, for recursively given algebraic series $g_i$ in $x$ and $z_1,...,z_i$, 
 
\[g_0(x)=\sigma(x)=\sum_{k=0}^\infty x^{p^k},\,g_i(x,z_1,...,z_i)=\sigma(z_i g_{i-1}^p(x,z_1,...,z_{i-1})).\] 
 
\paragraph{Structure of the paper.} 
In Section~\ref{sec:Fuchs} we revise the basic setup for solving linear differential equations in positive characteristic, using the space $\Rp$. Distinguished solutions, called xeric, will be constructed. They are characterized by having aside from the initial monomial no $p$\,th powers in their power series expansion. In this section, also variations and consequences of \cref{quest:algebraicity} will be discussed. 
 
Section~\ref{sec:expp} is concerned with the exponential function in positive characteristic, viz, the solution of $y'=y$. The respective xeric solution, denoted by $\expp$, is constructed according to the theory described in Section~\ref{sec:Fuchs}. It is then shown that for this solution, the projections $\expp\vert_{z_{i}=z_{i+1}=\cdots=0}$ are algebraic for all $i$. The proof makes a small detour: One shows  that another, specifically chosen solution $\widetilde{\exp}_p$ of $y'=y$, admits an infinite product decomposition with algebraic factors as mentioned above. And then a general result (see \ref{prop:monomially}) implies that also $\expp$ must have had algebraic projections. 
 
General first order differential equations will then be addressed in Section~\ref{sec:pk-curvature}, aiming at a proof of \cref{thm:introalg}. We will generalize the concept of $p$-curvature by introducing higher curvatures, called {\em $p^i$-curvatures}, for any $i\geq 1$. These curvatures share many properties with the classical $p$-curvature, but take into account the variables $z_i$ and yield finer information. This is then used in Section~\ref{sec:product} to set up the proof of \cref{thm:introalg}. In the final part, Section~\ref{sec:trig}, we investigate the solutions of the second order equations $y''=\pm y$ in order to compare the characteristic $p$ trigonometric functions with the exponential function. 
\goodbreak 
 
 
\section{On Fuchs' Theorem in Positive Characteristic} \label{sec:Fuchs} 
In this section, we first recall the main definitions and results from \cite{FH23}, reformulate them to fit our needs in this paper and to make \cref{quest:algebraicity} more precise. 
 
The theory of power series solutions of linear homogeneous differential equations in characteristic $p$ involving logarithms was initiated by Honda~\cite{Hon81}. Dwork \cite{Dwo90} studied the case of nilpotent $p$-curvature and F\"urnsinn and Hauser established the complete description of the solutions of arbitrary differential equations with regular singularities in \cite{FH23}. In all three cases one has to introduce certain differential extensions of $K\osb x\csb$.\\ 
 
Let us fix some notation. Let $p=\ch K$ be a prime number and 
\[L= a_n\partial^n+a_{n-1}\partial^{n-1}+\ldots + a_1\partial +a_0 \in K\osb x\csb[\partial]\] be a differential operator with power series coefficients $a_i\in K\osb x\csb$ over a field of characteristic $p$. We assume $L$ to be regular at $0$, i.e., that the quotient $a_i/a_n$ has a pole of order at most $n-i$ in $0$. Write $L=\sum_{j=0}^n\sum_{i=0}^\infty c_{ij}x^i\partial^j$ with $c_{ij}\in K$. We define the \textit{initial form} $L_0$ of $L$ as the operator 
\[L_0=\sum_{i-j=\tau}c_{ij}x^i \partial^j,\] 
where $\tau$ is the minimal \textit{shift} $i-j$ of $L$. We will restrict without loss of generality to differential operators with minimal shift $\tau=0$; this can be achieved by multiplication of $L$ with a suitable power of $x$. Consequently $L_0(x^k)=\chi_L(k)x^k$, where $\chi_L\in K[s]$ is the {\it indicial polynomial} of $L$; its roots $\rho$ are the \textit{local exponents} of $L$.\\ 
 
In this text we will assume for simplicity that $K=\F_p$ and that the local exponents belong to $\F_p$, i.e., that the indicial polynomial splits over $\F_p$. For general fields and local exponents, the theory extends as described in~\cite{FH23}, where the added difficulties are mostly being of technical and notational nature. \\ 
 
As explained in \cref{sec:intro}, define $\mathcal{R}_p\coloneqq \F_p(z_1,z_2,\ldots)\orb x\crb$ as the field of Laurent series in $x$ with rational functions in countably many variables $z_i$ as coefficients equipped with the aforementioned derivation. This derivation rule resembles the differentiation of the iterated (complex) logarithm $\log(\ldots \log(x)\ldots)$. We will therefore call the variables $z_i$ colloquially {\it logarithms}. The definition is motivated by the need to provide for any element of $\Rp$ a primitive under the derivation. For example, while $x^{p-1}$ does not have a primitive in $\F_p\osb x\csb$, we have $(x^pz_1)'=x^{p-1}\in \Rp$. 
More generally, any element of $\R_p$ admits its primitive in $\R_p$ as seen in the following lemma. 
We also reamrk that the field of constants of $\mathcal{R}_p$ turns out to be $\mathcal{C}_p\coloneqq \F_p(z_1^p,z_2^p,\ldots)\orb x^p\crb$ \cite[Prop.~3.3]{FH23}.

\newcommand{\T}{\mathcal{T}} 

\begin{lem}\label{lem:primitive}
For each polynomial $f\in \F_p[x, z_1,z_2,\ldots]$, there exists $F\in \F_p[x, z_1,z_2,\ldots]$ such that $F'=f$, and for each element $f\in \Rp$, there exists $F\in \Rp$ such that $F'=f$.
\end{lem}
\begin{proof}
First we consider the case $f\in\F_p[x, z_1,z_2,\ldots]$. 
We can restrict to the case that $f$ is a monomial 
$x^jz^\alpha=x^j\prod_{i=1}^\infty z_i^{\alpha_i}$, by linearity of differential operators. 
We proceed the proof by induction on 
$e(x^jz^\alpha)=\overline{j}+\sum_{i=1}^\infty\overline{\alpha_i} p^i\in\N$, 
where 
$\overline{\cdot}:\Z\to \{0, 1, \ldots, p-1\}$ denotes the residue map modulo $p$. 
When $e(x^jz^\alpha)=0$, we have $f=x^jz^\alpha\in\Cp$ and one primitive is $xf$.  

Assume $e(x^jz^\alpha)>0$.  We regard $x$ as $z_0$ and $j$ as $\alpha_0$.
We set $i_0=\min\{i\in\N\mid\overline{\alpha_i}\neq p-1\}$ 
and $g=\prod_{i>i_0}^\infty z_i^{\alpha_i}$.  Then we have 
$$
x^j z^\alpha=(xz_1\dotsm z_{i_0-1})^{p-1}z_{i_0}^{\alpha_{i_0}}g
=
({\alpha_{i_0}+1})^{-1}
(xz_1\dotsm z_{i_0-1})^{p}({z_{i_0}^{\alpha_{i_0}+1}})'g
$$
and, by integration by parts, we obtain 
$$
\int x^jz^\alpha=
({\alpha_{i_0}+1})^{-1}(xz_1\dotsm z_{i_0-1})^{p}
\left(
{z_{i_0}^{\alpha_{i_0}+1}}g-\int \Bigl({z_{i_0}^{\alpha_{i_0}+1}}g'\Bigr)
\right),
$$
which is homogeneous in $x$ of degree $j+1$. One easily checks that the monomials appearing in the expansion of 
$(xz_1\dotsm z_{i_0-1})^{p}z_{i_0}^{\alpha_{i_0}+1}g'$ have smaller values of $e$ than 
$e(f)$, so
the monomial $f=x^jz^\alpha$ has a primitive in $\F_p[x, z_1,z_2,\ldots]$
by the induction hypothesis. 

Next we consider the case $f\in \R_p$. Again, by linearity, it suffices to restrict to homogeneous parts with respect to $x$, i.e., $f=r(z)x^k$, where $r(z)=a(z)/b(z)$ for some polynomials $a(z), b(z)\in \F_p[z_1,z_2,\ldots]$. Upon multiplication by the constant $b(z)^p$, we might assume that $r(z)$ is a polynomial itself. Now the same argument as before shows that a primitive $F$ of the element $f$ exists in $\Rp$. 
\end{proof}

Every differential operator $L\in \F_p\osb x\csb[\partial]$ defines a $\Cp$-linear map 
\[L:\Rp\to \Rp,\quad y\mapsto L(y),\] 
applying $L$ to series $y\in \Rp$. 
Similarly, its initial form $L_0$ and its tail $T=L_0-L$ define $\Cp$-linear maps. We represent the local exponents by integers between $0$ and $p-1$. With this convention, it is is easy to see that the monomial 
\[x^\rho z^{i^*},\quad \text{for $\rho$ a local exponent and $0\leq i\leq m_\rho-1$,} \] 
with exponents $i^*\in \N^{(\N)}$ defined by 
\[i^*=(i, \lfloor i/p \rfloor, \lfloor i/p^2 \rfloor,...),\] 
form a monomial $\Cp$-basis of the kernel $\Ker\, L_0$ of $L_0$, i.e., of the solution space of the Euler equation $L_0y=0$ in $\Rp$ \cite[Prop.~3.9]{FH23}. Here $z^\alpha$ for $\alpha\in \N^{(\N)}$ denotes $z_1^{\alpha_1}\cdots z_k^{\alpha_k}$ where $k\in \N$ is maximal such that $\alpha_k\neq 0$. \\ 
 
To formulate Fuchs' Theorem in positive characteristic, it is convenient to choose a direct complement $\mathcal{H}$ of $\Ker\, L_0$ in $\Rp$ as a $\Cp$-vector space. There are several choices for $\mathcal{H}$, and we will discuss one particular below. The restriction ${L_0\vert}_{\mathcal{H}}$ of $L_0$ to $\mathcal{H}$ defines an isomorphism onto the image, which is shown to be again $\R_p$, using the fact that the differential field $\R_p$ contains sufficiently many primitives. Let $S: \R_p\to \mathcal{H}$ be the inverse of ${L_0\vert}_{\mathcal{H}}$, i.e., a section (or right inverse) of $L_0$. We get a $\Cp$-linear map 
\[v: \Rp \to \mathcal{H},\quad y\mapsto v(y)=\sum_{k=0}^\infty (ST)^k(y).\] 
It is well defined because the composition $ST=S\circ T$ increases the order in $x$ of a series in $\Rp$, thus $\sum_{k=0}^\infty (ST)^k(y)$ converges to a formal series. 
 
In this setting, one has the following extension of Fuchs' Theorem to the case of linear differential equations defined over a field of positive characteristic. We give here a simplified version, for the general statement, see \cite[Thm.~3.16, Thm.~3.17]{FH23} 
 
\begin{thm}[Fuchs' Theorem in Positive Characteristic] \label{thm:Fuchs} Let $L\in K\osb x\csb[\partial]$ be a differential operator with shift $0$. Decompose $L=L_0-T$ into its initial operator $L_0\in K[x] [\partial]$ and tail operator $T\in K[x] [\partial]$. Let $\mathcal{H}$ be a direct complement of $\Ker\, L_0$ in $\Rp$, and let $S=({L_0\vert}_{ \mathcal{H}})^{-1}$ be defined as before. 
 
\begin{enumerate}[(i)] 
\item Assume that $L$ has local exponent $\rho\in\F_p$ at $0$. Then 
\[y(x)=v(x^\rho)=\sum_{k=0}^\infty (ST)^k(x^\rho)\in \Rp\] 
is a solution of $Ly=0$. 
\item Assume that $L$ has a regular singularity at $0$ and that all local exponents of $L$ are in $\F_p$. Then the series $y_{\rho, i}(x)=v(x^\rho z^{i^*})\in\R_p$ form a $\Cp$-basis of solutions of $Ly=0$. Here, $\rho$ ranges over all local exponents, $m_\rho$ denotes their multiplicity, and $0\leq i< m_\rho$. 
\end{enumerate} 
\end{thm} 
 
By abuse of notation we have written $\rho$ for the local exponent of $L$ in $\F_p$, as well as for its representative in $\{0, 1, \ldots, p-1\}\subseteq \Z$. 

The theorem extends results of \cite{Hon81, Dwo90}: Honda considered equations with zero $p$-curvature, in which case no extra variables $z_i$ are required to describe the solutions. Dwork treated more generally the case of nilpotent $p$-curvature, and there finitely many $z_i$ suffice to get a basis of solutions. For arbitrary equations with regular singularities, infinitely many $z_i$-variables may be necessary if the $p$-curvature is not nilpotent.\\

For $j\in \F_p$ and $\gamma\in \F_p^{(\N)}$ we define the \textit{section operators} $\langle \cdot \rangle_{j, \gamma}:\R_p\to \R_p$ by extracting those monomials $x^kz^\alpha$ of the expansion of an element $f\in \Rp$ for which $k\equiv j\bmod p$ and $\alpha_i\equiv \gamma_i\bmod p$ for all $i$. More explicitly, 
\[\left \langle \sum c_{k, \alpha}x^k z^\alpha\right \rangle_{\! j, \gamma}\coloneqq\sum_{\mathclap{\substack{k\in j+p\Z\\ \alpha\in \gamma+(p\Z)^{(\N)}}}} c_{k, \alpha}x^k z^\alpha.\] 
Note that $x^{-j}z^{-\gamma}\langle y\rangle_{j,\gamma}\in \Cp$.\\ 
 
By Theorem~\ref{thm:Fuchs}, any regular singular differential equation $Ly=0$ admits a $\Cp$-basis of solutions, and in part (ii) of the theorem the construction of a specific basis is described in terms of an algorithm. It turns out that the resulting basis can be described intrinsically by conditions on the exponents of the involved series $y_{\rho, i}$. This works as follows. 
 
A solution $y=y_{\rho, i}\in \R_p$ of $Ly=0$  will be called \textit{xeric} if there is a local exponent $\rho$ of $L$ and an index $0\leq i<m_\rho$ such that 
\[\langle y_{\rho, i} \rangle_{\rho,i^*}=x^\rho z^{i^*}\] 
  and 
\[\langle y_{\rho, i} \rangle_{\sigma,j^*}=0\] 
 for all pairs $(\sigma, j)\neq (\rho, i)$ of local exponents $\sigma$ and indices $0\leq j<m_\sigma$. This signifies that aside from the \textit{initial monomial} $x^\rho z^{i^*}$ there occurs no $p$th power multiple of some $x^\sigma z^{j^*}$ in the expansion $\sum c_{k, \alpha}x^k z^\alpha$ of $y$. This description explains the naming in the sense of ``deprived of". Bases of xeric solutions of differential equations with regular singularities always exist and are then unique. In fact, it suffices to apply Theorem~\ref{thm:Fuchs} in the case where the direct complement $\mathcal{H}$ of $\Ker \, L_0$ is chosen such that the power series expansion of any $y\in \mathcal{H}$ involves none of the monomials generating $\Ker \, L_0$. 
 
For the case of first order  operators with local exponent $\rho=0$ (necessarily of multiplicity $1$, hence $i=0$ and also $i^*=0$), the xeric solution is the unique solution $y$ whose expansion involves no $p$th power monomial except for the constant term $1$. 
 
\begin{ex} 
We consider the series $\exp_p$, $\log_p(1-x)$, $\sin_p(x)$, and $\cos_p(x)$ in the characteristic $p$ setting, that is, the \monominduced solutions of 
\[y'=y, \quad xy''-y'-x^2y''=0, \quad \text{and} \quad y''=-y,\] 
respectively. Taking $p=3$, one obtains 
\begin{align*} 
&\exp_3 = 1 + x + 2x^2 + 2x^3z_1 + (2z_1 + 1)x^4 + x^5z_1 + 2x^6z_1^2 + (2z_1^2 + 2z_1 + 1)x^7 \\ &\phantom{\exp_3 =1\,} 
+(z_1^2 + 2)x^8 + (z_1^3z_2 + 2z_1)x^9 + (z_1^3z_2 + 2z_1^2 + z_1 + 2)x^{10}+\ldots,\\ 
&\log_3(1-x) = x + 2x^2 + x^3z_1,\\ 
&\sin_3 = x+2z_1x^3+z_1 x^5+(2 z_1^2+2 z_1) x^7+z_1^3 z_2 x^9+(2 z_1^3 z_2+z_1) x^{11}+\ldots,\\ 
&\cos_3 = 1 + 2x^2 + 2x^4z_1 + (2z_1^2 + z_1)x^6 + (z_1^2 + 2z_1 + 2)x^8 +\ldots 
\end{align*} 
No obvious pattern seems to be recognizable. 
\end{ex}\goodbreak 
 
The coefficients $c_k$ of Laurent series $\sum c_k(z)x^k$ in $\Rp$ are rational functions in the variables $z_1, z_2, \ldots$ As such, each of them depends only on finitely many variables (by definition of polynomials and rational functions in infinitely many variables), but this number may increase with the exponent $k$ of $x$ and actually may go to $\infty$. We will be interested in series as above which involve only \textit{finitely many} $z$-variables, that is, in the subrings 
\[ 
\Rp^{(i)} \coloneqq \F_p(z_1, \ldots, z_i)\orb x\crb 
\] 
of $\Rp$.  
Restricting to polynomial coefficients in $z$ and setting $z_{i+1}= z_{i+2} = \ldots = 0$ we get projection maps 
\begin{gather*} 
\pi_i : \F_p[z_1,z_2,\ldots]\orb x\crb \rightarrow \F_p[z_1, z_2,\ldots, z_i]\orb x\crb\subseteq \Rp^{(i)},\\ 
 y(x, z) \mapsto y(x, z)\vert_{ z_{i+1}= z_{i+2} = \ldots = 0} = y(x, z_1, \ldots, z_i, 0, \ldots, ). 
\end{gather*} 

The following lemma is easy but important.
\begin{lem}\label{lem:algandproj}
For $f\in\F_p[z_1,z_2,\ldots]\orb x\crb$, the following are equivalent:
\begin{enumerate}[(i)] 
 \item 
$f$ is algebraic over $\F_p(x,z_1,\dotsc)$,
\item
$f\in\Rp^{(k)}$ and 
$f$ is algebraic over $\F_p(x,z_1,\dotsc,z_k)$ 
for some $k\in\N$,
\item
$f=\pi_k(f)$ and 
$\pi_k(f)$ is algebraic over $\F_p(x,z_1,\dotsc,z_k)$ 
for some $k\in\N$.
\end{enumerate}
\end{lem}
\begin{proof}
Assume (i) and let 
$P(T)\in\F_p(x,z_1,\dotsc)[T]$
be a minimal polynomial of $f$.  
Since the coefficients of $P$ contain only finitely many variables $z_i$, 
we may take $k\in\Z_{\geq0}$ such that 
$P(T)\in\F_p(x,z_1,\dotsc,z_k)[T]$.
As $z_i$ for $i>k$ is transcendental over $\F_p(x,z_1,\dotsc,z_k)$, 
it cannot appear in the expansion of $f$.  
Therefore (ii) holds. 
The other equivalences are clear.
\end{proof}
 
\begin{rem}\label{rem:primitive} 
In contrast to characteristic $0$, algebraicity is preserved under taking primitives in $\Rp$. For example, for every algebraic function $h\in \F_p\orb x\crb$ there exist primitives having algebraic projections: To see this, consider $f = \int h$,which is unique up to the addition of a $p$th power. Expand $h$ uniquely into $h = \sum_{k=0}^{p-1} h_k(x^p) x^i$, for some $h_k\in \F_p\orb x \crb$, and we will see that they are algebraic in Lemma~\ref{lem:secalg}. Then 
 
\[f = \sum_{k=0}^{p-1} h_k(x^p) \int x^k = \sum_{k=0}^{p-2} h_k(x^p) \frac{1}{k+1} x^{k+1} + h_{p-1}(x^p) x^p z_1\] 
is a primitive and again algebraic. The argument easily extends to $f\in \Rpk$, which are algebraic over $\F_p(x, z_1,\ldots, z_k)$.
\end{rem} 
 
\begin{rem} 
It turns out that in Fuchs' Theorem~\ref{thm:Fuchs} one may specify more accurately the subspace of $\F_p(z)\orb x\crb$ in which the solutions live. To this end, introduce, for every $k\geq 0$, the monomials 
 
\[ 
w_k \coloneqq z_1^{p^{k-1}}z_2^{p^{k-2}}\cdots z_{k-1}^p z_k^1. 
\] 
Thus, $w_1=z_1$, $w_2=z_1^p z_2$, $w_3=z_1^{p^2}z_2^pz_3$, and so on.  It was shown in \cite{FH23} that a basis of solutions of $Ly=0$ already exists in the subspace $\bigoplus_\rho x^\rho\F_p[w_1, w_2, w_3,\ldots] \osb x\csb$ where $\rho$ runs over the set of local exponents. Actually, one might restrict this space even further by bounding the degree of the variables $w_k$ in each monomial in terms of the degree of $x$.\\ 
 
For first order differential equations with local exponent $\rho=0$ this corresponds to the following construction: 
Define the ring 
 
\[ 
\Sp = \F_p\{x, x^pw_1, x^{p^2}w_2, \ldots\} 
\] 
as the closure of $\F_p[x, x^pw_1, x^{p^2}w_2, \ldots]$ in $\F_p(z_1,z_2,\ldots)\orb x\crb$ with respect to the $x$-adic topology. 
For example, infinite sums of the form 
 
\[ 
\sum_{k=0}^\infty b_k x^{p^k}w_k = \sum_{k=0}^\infty b_k x^{p^k}z_1^{p^{k-1}}z_2^{p^{k-2}}\cdots z_{k-1}^p z_k^1 
\] 
belong to $\Sp$, for any $b_k\in \F_p$, since the sum converges to a series in $\F_p(z_1,z_2,\ldots)\orb x\crb$. This will be illustrated in later sections. It is easy to check via generators that the ring $\Sp$ is differentially closed, in contrast to the ring $\F_p[z_1,z_2,\ldots]\osb x\csb$, in which we have $\partial z_1=x^{-1}$. 
 
As before, we may need to restrict to finitely many $z$-variables, and thus we set 
\[ 
\Spk = \F_p\{x, x^pw_1, x^{p^2}w_2, \ldots, x^{p^k}w_k\} = \Sp \cap \Rp^{(k)}. 
\] 
\end{rem}

In the following paragraphs, we want to make Question \ref{quest:algebraicity} more precise. Recall that the solutions of $Ly=0$ form an $n$-dimensional $\Cp$-vector space. In particular, if $y$ is a solution such that its initial series is algebraic, multiplying $y$ by a transcendental power series in $x^p$ gives another solution of $Ly=0$ whose initial series cannot be algebraic. However, it turns out that if a given differential equation $Ly=0$ has a basis of power series solutions with algebraic projections, the same holds true for its \monominduced basis. 
 
\begin{prop}\label{prop:monomially} 
Let $L\in \F_p\osb x\csb[\partial]$ be a differential operator of order $n$ and 
assume there is a basis 
$\widetilde y_1,\ldots,\widetilde y_n\in \F_p[z_1, z_2,\ldots]\osb x\csb$ 
of solutions of $Ly=0$ whose projection $\pi_k(\widetilde y_j)$ is algebraic over 
$\F_p(x, z_1,\ldots, z_k)$ for any $k$. 
Then the \monominduced basis $y_1,\ldots, y_n$ has algebraic projections for all $k$ as well. 
\end{prop} 
 
This result will be used later in Corollary \ref{cor:xericalgebraic} to establish the algebraicity of the xeric solution of a first order equation. 
We need the following lemma, generalizing a well known fact about sections of algebraic power series in characteristic $p$. 
 
\begin{lem} \label{lem:secalg} 
Let $f\in \F_p(z_1, z_2,\ldots, z_k)\orb x\crb$. Then $f$ is algebraic over $\F_p(x,z_1, z_2,\allowbreak \ldots, z_k)$ if and only if $\langle f \rangle_{j, \alpha}$ is algebraic for all $j\in \F_p, \alpha \in \F_p^{k}$. 
\end{lem} 
\begin{proof} 
Since $f$ is the finite sum over all its sections, the condition is sufficient.
To see that it is necessary, we use the induction on 
$\ell(f)=\min\{i\in\N\mid f^{(i)}=0\}$. This function $\ell$ is well-defined 
since $f^{(p^{k+1})}=0$, which will be proved in Lemma \ref{lem:derivationrules}.
When $\ell(f)=0$, we see that $f=0$ and 
all its sections are $0$. Thus the assertion holds. 

Let $f\in\R_p$ be algebraic and assume now that $\ell(f)>0$.  Then $f'$ is also algebraic and 
$\ell(f')=\ell(f)-1$. Thus, by the induction hypothesis, we can decompose $f'$ 
into algebraic sections, that is, there exist algebraic elements 
$g_{j,\alpha}\in\R_p$ for $j\in \F_p$ and $\alpha \in \F_p^{k}$ 
such that $f'=\sum_{j,\alpha}g_{j,\alpha}^px^jz^\alpha$.
Set 
$g=\sum_{j,\alpha}g_{j,\alpha}^p\int(x^jz^\alpha)$, in accordance with Lemma \ref{lem:primitive}.
Note that the sections of $g$ are algebraic since they are 
$\F_p[x,z_1,\dotsc]$-linear combinations of the elements $g_{j,\alpha}^p$.  
We set $h=f-g$. Since $f$ and $g$ are both algebraic, so is $h$.  
Since $h'=f'-g'=0$, we see that $h\in\Cp$ and 
$h=\langle h \rangle_{0, \bf{0}}$.  
Thus sections of $g$ and $h$ are algebraic, and hence the sections of 
$f$ are algebraic too.
\end{proof} 

\begin{proof}[Proof of Proposition~\ref{prop:monomially}.] 
Let $Y=(y_1,\ldots,y_n)^\top$ and 
$\widetilde Y=(\widetilde y_1,\ldots, \widetilde y_n)^\top$.  
Then there is an invertible matrix 
$C\in \operatorname{GL}_n(\Cp)$ such that $C Y=\widetilde Y $.
Since $y_1,\ldots,y_n$ is a xeric basis, there exist $n$ distinct 
pairs $(\rho_j,i_j)$ for $1\leq j\leq n$ such that 
$\rho_j$ is a local exponent of $Ly=0$ with multiplicity $m_j$, 
$1\leq i_j\leq m_j$ is an index, and 
$\langle y_{j'}\rangle_{\rho_j,i_j^\ast}=\delta_{j,j'}x^{\rho_j}z^{i_j^\ast}$.
It follows that 
$$
\langle Y\rangle_{\rho_j,i_j^\ast}
=(\langle y_1\rangle_{\rho_j,i_j^\ast},\ldots,
\langle y_n\rangle_{\rho_j,i_j^\ast})^\top=
x^{\rho_j}z^{i_j^\ast}e_j,
$$
where $e_j$ denotes the $j$-th unit vector. 

We obtain 
\[\langle \widetilde Y\rangle _{\rho_j, i^*}=\langle C Y\rangle_{\rho_j,i^\ast}= C x^{{\rho_j}}z^{i^\ast} e_{j},\]  
i.e, the entries of $C$ are essentially given by the sections of $\widetilde Y$. Take $N\in\Z_{>0}$ in such a way that 
$x^{\rho_j}z^{i_j^\ast}=\pi_N(x^{\rho_j}z^{i_j^\ast})$ 
for any $1\leq j\leq n$. 
Since $\det C\neq0$, changing $N$ if necessary, 
we may also assume $\pi_N(\det C)\neq0$.

Let $k\geq N$.  
Since the entries of $\pi_k(\widetilde{Y})$ are algebraic 
by the assumption, those of 
$\langle\pi_k(\widetilde{Y})\rangle_{\rho_j,i_j^\ast}$ 
are also algebraic by Lemma \ref{lem:secalg}. Note that 
$$
x^{\rho_j}z^{i_j^\ast}\pi_k(C)e_j
=\langle \pi_k(\widetilde{Y})\rangle_{\rho_j,i_j^\ast}.
$$
Therefore we conclude that all entries of $\pi_k(C)$ are algebraic.
Since $\det \pi_k(C)\neq0$, we also see that 
$\pi_k(C)\in\operatorname{GL}_n(\Cp)$.
It follows that all the entries of 
$\pi_k({Y})=\pi_k(C)^{-1}\pi_k(\widetilde{Y})$ 
are algebraic, since those of
$\pi_k(C)$ and $\pi_k(\widetilde{Y})$ are algebraic.
Thus $\pi_k(y_j)$ is algebraic for all $k\geq N$.  
Since $\pi_N(y_j)$ is algebraic, it is immediate that 
$\pi_k(y_j)$ is algebraic for all $k\leq N$.
\end{proof} 
 
Thus, we have essentially reduced Problem~\ref{quest:algebraicity} to the following. 
 
 
\begin{prob}\label{quest:algmonom} 
For which differential operators $L\in \F_p[x][\partial]$ does the basis of \monominduced solutions of $Ly=0$ have algebraic projections? 
\end{prob} 
 
We conclude this section with an assertion about the algebraicity of projections: 
\begin{lem} \label{lem:zcoeff} 
Let $f\in \F_p[z_1, z_2,\ldots, z_k]\osb x\csb$ be algebraic over $\F_p(x,z_1, z_2,\ldots, z_k)$. Write $f=\sum f_\alpha z^\alpha$ for $f_\alpha \in  \F_p\osb x\csb$. Then $f_\alpha$ is algebraic over $\F_p(x)$ for all $\alpha\in \N^{k}$. 
\end{lem}

\begin{proof} 
We use induction on the number of variables $k$. The case $k=0$ is trivial, so assume we have proven the statement for $k-1$. Take $f\in \mathbb{F}_p[z_1,\ldots,z_k]\osb x\csb$ algebraic over $\mathbb{F}_p(x,z)$ with minimal polynomial $P.$ Chose $\alpha=(\alpha', \alpha_n)\in \N^k$ with $\alpha'\in \N^{k-1}$. Setting $z_k=0$ in the identity $P(f)=0$ shows that $f\vert_{z_k=0}\in\mathbb{F}_p[z_1,\ldots,z_{k-1}]\osb x\csb$ is algebraic over $\mathbb{F}_p(x,z).$ 
Then the induction hypothesis applies so we know that $\left(f\vert_{z_k=0}\right)_{\alpha'}=f_{(\alpha', 0)}$ is algebraic over $\mathbb{F}_p(x,z)$. We can apply this argument to the algebraic element $(f-f\vert_{z_k=0})/z_k$ to get algebraicity of $f_{(\alpha', 1)}$ and repeat to show the algebraicity of $f_\alpha=f_{(\alpha',\alpha_n)}$. 
\end{proof} 
 
 
\section{The Exponential Differential Equation}\label{sec:expp} 
In \cite{FH23} the exponential function $\exp_p$ in characteristic $p$ was defined as the \monominduced solution of $y'=y$. All further solutions of $y'=y$ in $\mathcal{R}_p$ are then given by $\mathcal{C}_p$-multiples of $\exp_p$. In this section we are going to define a different element $\widetilde{\exp}_p\in \mathcal{R}_p$ as an infinite product with specified factors and show that it is another solution of $y'=y$.\\ 
 
We start with some preliminaries. 
Recall that 
$w_k= z_1^{p^{k-1}}z_2^{p^{k-2}}\cdots z_{k-1}^p z_k^{1}=w_{k-1}^pz_k$. 
Then clearly 
\[ (x^{p^k}w_k)'=x^{p^k-1} z_1^{p^{k-1}-1}z_2^{p^{k-2}-1}\cdots z_{k-1}^{p-1} 
=x^{p-1}(x^pw_1)^{p-1}\cdots (x^{p^{k-1}}w_{k-1})^{p-1} 
.\] 
Let $M$ be a monomial in $\F_p[x,z_1, z_2,\ldots, z_k]$, for some $k\geq 0$. 
Then $M^{(p^{k+1})}=0$ vanishes for any such $M$, and 
$M=(x^{p^k}w_k)'$ is up to multiplication with elements of $\mathcal{C}_p$ the unique element such that $M^{(p^{k+1}-1)}\neq0$.
More precisely, we have the following result.

\begin{lem}\label{lem:derivationrules} 
The derivation 
on $\Rp$ satisfies the following rules. 
\begin{enumerate}[(i)] 
\item 
$\ds (x^{\alpha_0}z_1^{\alpha_1}\dotsm z_k^{\alpha_k})^{(n)}=0$ 
\quad for $n\geq p^{k+1}$. 
\item 
$\ds z_{k+1}^{(p^{k+1})}
=(x^{-1}z_1^{-1}\dotsm z_k^{-1})^{(p^{k+1}-1)}
=(-1)^{k+1}x^{-p^{k+1}}z_1^{-p^{k}}\dotsm z_k^{-p}$. 
\item 
$\ds 
(x^{\alpha_0}z_1^{\alpha_1}\dotsm z_k^{\alpha_k})^{(p^{k+1}-1)}
\neq0$ if and only if $\alpha_i \in p{\mathbb Z}-1$ for all $i$. 
\item 
$\left(
(x^{p^{k+1}}w_{k+1})'
\right)^{(p^{k+1}-1)}
=(-1)^{k+1}$. 
\end{enumerate} 
\end{lem} 
\begin{proof} 
We regard $x$ as $z_0$. First we prove (i) and (ii) by induction. 
They are clear for $k=-1$ since 
$(1)^{(n)}=0$  for $n\geq 1$ and 
$(x)'=1=(-1)^{-1+1}$. 
Assume (i) and (ii) holds up to $k-1$. 
We show (i) for $k$. 
Note that we may assume $0\leq \alpha_i\leq p-1$. 
By the product rule, we have 
\begin{equation} \label{eq:monomials} 
(x^{\alpha_0}z_1^{\alpha_1}\dotsm z_k^{\alpha_k})^{(n)}
=\sum_{j,\ell}\sum_{i_{j,\ell}} 
\dfrac{n!}{i_{0,0}!\dotsm i_{k,\alpha_k}!} 
\prod_{j=0}^k 
\prod_{\ell=1}^{\alpha_j} 
z_j^{(i_{j,\ell})}, 
%
\end{equation} 
where $0\leq j\leq k$, $1\leq \ell\leq \alpha_j$, and $i_{j,\ell}\geq 0$ with $\sum_{\ell=1}^{\alpha_j} i_{j,\ell} =n$.
If this sum is non-zero, we have 
$i_{j,\ell}\leq p^{j}$ for any $j$ and $\ell$ by (ii), thus we have 
\begin{equation} \label{eq:inequality} n\leq 
\sum_{j=0}^k 
\sum_{\ell=1}^{\alpha_j} 
p^j 
=\sum_{j=0}^k\alpha_jp^j 
\leq\sum_{j=0}^k(p-1)p^j\leq p^{k+1}-1<p^{k+1}, 
\end{equation}
which verifies (i) for $k$.  By Equation (\ref{eq:monomials}), we also have
\begin{align*} 
&(xz_1\dotsm z_{k})^{p} 
z_{k+1}^{(p^{k+1})}
=(xz_1\dotsm z_{k})^{p} 
(x^{-1}z_1^{-1}\dotsm z_{k}^{-1})^{(p^{k+1}-1)}
\\& 
=\left((xz_1\dotsm z_{k})^{p-1}\right)^{(p^{k+1}-1)}
= 
\dfrac{(p^{k+1}-1)!}{ 
(1!p!\dotsm(p^k)!)
^{p-1} 
} 
\left(\prod_{j=0}^kz_j^{(p^j)}\right)^{p-1}. 
\end{align*} 
Now by Lucas's formula we obtain 
\begin{align*} 
\dfrac{(p^{k+1}-1)!}{ 
(1!p!\dotsm(p^k)!)^{p-1} 
} 
= 
\prod_{s=0}^k 
\prod_{t=1}^{p-1} 
\binom{t\cdot p^{s}+p^s-1}{p^s} 
=((p-1)!)^{k+1}=(-1)^{k+1}. 
\end{align*} 

Also by (ii) of the induction hypothesis, we have 
\begin{align*} 
\left(\prod_{j=0}^kz_j^{(p^j)}\right)^{p-1} 
& 
=\left(
\prod_{j=0}^k 
(-1)^{j}x^{-p^{j}}z_1^{-p^{j-1}}\dotsm z_{j-1}^{-p} 
\right)^{p-1} 
=\dfrac{(xz_1\dotsm z_{k-1})^p} 
{x^{p^{k+1}}z_1^{p^{k}}\dotsm z_{k-1}^{p^2}}. 
\end{align*} 
Therefore we conclude that 
$$ 
z_{k+1}^{(p^{k+1})}
=\dfrac{(-1)^{k+1}} 
{x^{p^{k+1}}z_1^{p^{k}}\dotsm z_{k-1}^{p^2} 
z_{k}^{p}}, 
$$ 
which proves (ii) for $k$. 
 
\medskip 
 
The rest is easy. For (iii), the forward implication follows from  
inequality (\ref{eq:inequality}), while 
the backward implication follows from (ii). And (iv) is equivalent to (ii). 
\end{proof} 
 
Higher derivatives of $w_k$ are in general 
sums of monomials without any obvious pattern. 
Thus the following refinement of Lemma \ref{lem:derivationrules} (iv) 
is quite surprising. 
 
\begin{prop} \label{prop:expperiodicity} 
The derivatives $(x^{p^{k}} w_{k})'$ of the monomials $x^{p^{k}} w_{k}$ 
satisfy the differentiation rule 
\[\left(
(x^{p^{k+1}} w_{k+1})'\right)^{(p^{k+1}-p^{k})} 
 = - (x^{p^{k}} w_{k})'.\] 
\end{prop} 
 
The formula will be a consequence of Theorem~\ref{thm:expt} and Proposition~\ref{prop:period}. One can also give a proof by direct computation.\\ 
 
This proposition suggests an alternative construction of a solution of $y'=y$: Note first that a series $y=\sum_{i=0}^\infty a_i(z)x^i\in \Rp$ with $a_i\in \F_p\orb z_1,z_2,\ldots\crb $ is a solution of $y'=y$ if and only if 
\begin{enumerate}[(i)] 
\item $ a_0'=0,$ i.e., $a_0\in \F_p(z_1^p,z_2^p,\ldots)$, and 
\item $(a_i(z) x^i)' = a_{i-1}(z) x^{i-1}$. 
\end{enumerate} 
So we set $a_{p^k-1}(z)\coloneqq (-1)^k w_k'$ for all $k$ and then define $a_{p^{k}-m}(z)$ via the equations  $x^{p^{k}-m}a_{p^{k}-m}(z)=  (-1)^k (x^{p^k}w_k)^{(m)}$ for all $m\geq 1$. Proposition~\ref{prop:expperiodicity} shows that the resulting series is well-defined and a solution of $y'=y$. In the next paragraphs we show that it coincides with a solution $\widetilde \exp_p$ of $y'=y$ that can be defined by completely different means. This other definition is less intuitive, but will prove to be more convenient for our calculations. Proposition \ref{prop:period}  shows that the two solutions agree.\\ 
 
We define the continuous $\F_p$-automorphism \[\sigma:\F_p\osb t\csb\to \F_p\osb t\csb,\quad t\mapsto \sum_{k=0}^\infty t^{p^k},\] and set recursively 
\[g_0\coloneqq \sigma(x)\quad \text{and} \quad g_{i+1}\coloneqq \sigma(g_i^pz_{i+1}).\] 
 
Further we define the polynomial 
\[H(t)\coloneqq \prod_{k=1}^{p-1}\left(1-\frac{t}{k}\right)^k, \text{ for $t$ a variable}\] 
and the series 
\[\widetilde{\exp}_p(x,z)\coloneqq \prod_{i=0}^\infty H\left((-1)^ig_i\right)\in \Rp.\] 
Note that $\widetilde{\exp}_p$ is well-defined as an element of $\Rp$, because $g_i\in x^{p^i}\F_p[z_1,z_2,\ldots]\osb x\csb$. Clearly $\widetilde{\exp}_{p}\vert_{x=0}=1$. To show that $\widetilde{\exp}_p$ is indeed a solution of $y'=y$, we resort to the logarithmic derivative of $H$. 
 
\begin{lem}\label{lem:magic} 
The polynomial $H$ is a solution of $(1-t^{p-1})y'=y$, say, has logarithmic derivative 
\[\frac{H(t)'}{H(t)} =\frac{1}{1-t^{p-1}}=\sum_{k=0}^\infty t^{k(p-1)}.\] 
In particular, for $\sigma$ and $g_i$ as defined above, 
\[ \frac{H(\sigma(t))'}{H(\sigma(t))} =\frac{\sigma(t)}{t} \quad  \text{and} \quad \frac{ H((-1)^ig_i)'}{H((-1)^ig_i)}=\frac{(-1)^ig_i}{xz_1\cdots z_{i}}.\] 
\end{lem} 
\begin{proof} 
By the additivity of the logarithmic derivative 
\[\frac{(fg)'}{fg}=\frac{f'}{f}+\frac{g'}{g},\] we have 
\begin{equation}\label{eq:logderH} 
\frac{H(t)'}{H(t)}=-\sum_{k=1}^{p-1} \frac{1}{1-\frac{1}{k}t} 
.\end{equation} 
Let $s$ be another variable and set \[F(s)=\prod_{k=1}^{p-1}\left (s+1-\frac{t}{k}\right)=\sum c_i(t)s^i.\] 
Using Fermat's Little Theorem we see that $F(s)=(s+1)^{p-1}-t^{p-1}$ as their zero sets agree. In particular, we have $c_0(t)=1-t^p$ and $c_1(t)=-1$. Thus, 
we further obtain, bringing \eqref{eq:logderH} to a common denominator, 
 
\[\frac{H(t)'}{H(t)}=-\frac{c_1(t)}{c_0(t)}=\frac{1}{1-t^{p-1}}.\] 
So for $H(\sigma(t))$ we obtain 
\[\frac{H(\sigma(t))'}{H(\sigma(t))}=\frac{\sigma(t)'}{1-\sigma(t)^{p-1}}=\frac{\sigma(t)}{\sigma(t)-\sigma(t)^p}=\frac{\sigma(t)}{t}. \] 
In this identity, setting $t=(-1)^{i}g_{i-1}^pz_i,$ i.e., $\sigma(t)=(-1)^ig_i$, we obtain \[\frac{1}{t}=\frac{z_i'}{z_i}=\frac{1}{xz_1\cdots z_i}\] and 
\[ 
\frac{H((-1)^ig_i)'}{H((-1)^ig_i)}=\frac{(-1)^{i}g_i}{xz_1\cdots z_i}. 
\qedhere \] 
\end{proof} 
 
\begin{rem} 
The identity 
\[\frac{H(t)'}{H(t)}=\frac{1}{1-t^{p-1}}\] 
can also be derived from 
\[\frac{H(t)'}{H(t)}=-\sum_{k=1}^{p-1} \frac{1}{1-\frac{1}{k}t} 
 =-\sum_{i=0 }^\infty t^i \sum_{k=1}^{p-1}\frac{1}{k^i}\] using the well-known fact 
\[\sum_{k=1}^{p-1}k^i \equiv \begin{cases} -1 & \text{if } i\equiv 0 \bmod p-1\\ 0 & \text{else} \end{cases}\quad \bmod p.\] 
The formula for the sum of the first $n$ of the $i$th powers $k^i$ in terms of Bernoulli numbers is called Faulhaber's formula, who computed the sums for the first 17 values of $i$ in \cite{Fau31} in the early $17$th century. The above-mentioned fact is an easy corollary. 
\end{rem} 
 
\begin{thm} \label{thm:expt} 
The series $\widetilde{\exp}_p(x,z)=\prod_{i=0}^\infty H\left((-1)^ig_i\right)\in \Sp$ is an exponential function in characteristic $p$, 
\[(\widetilde{\exp}_p)'=\widetilde{\exp}_p.\] 
\end{thm} 
 
\begin{proof} 
By the additivity of the logarithmic derivative and Lemma~\ref{lem:magic} we have 
\[\frac{(\widetilde \exp_p)'}{\widetilde \exp_p}=\sum_{i=0}^\infty \frac{(-1)^i g_i}{xz_1\cdots z_i}.\] 
We will show inductively that 
\[\sum_{i=0}^k \frac{(-1)^i g_i}{xz_1\cdots z_i}= 1+\frac{(-1)^kg_k^p}{xz_1\cdots z_{k}}.\] 
Then, as $g_k\in x^{p^k}\F_p(z_1, z_2,\ldots )\osb x\csb,$ it will follow that 
\[\frac{(\widetilde \exp_p)'}{\widetilde \exp_p}-1 =\lim_{k\to \infty}\frac{(-1)^kg_k^p}{xz_1\cdots z_{k}} \in \bigcap_{k=0}^\infty x^{p^k-1}\F_p(z_1, z_2,\ldots)\osb x\csb=0.\] 
 
The induction is straightforward: For $k=0$, and using that $g_0^p=g_0-x$, the claim holds. Moreover, 
\[1+\frac{(-1)^kg_k^p}{xz_1\cdots z_{k}}+\frac{(-1)^{k+1}g_{k+1}}{xz_1\cdots z_{k+1}}=1+\frac{(-1)^{k+1}g_{k+1}^p}{xz_1\cdots z_{k+1}},\] 
as $g_k^pz_{k+1}=g_{k+1}-g_{k+1}^p.$ 
\end{proof} 
 
\begin{rem} 
The definition of $\exppt$ as an infinite product can be motivated as follows: We want to find an element of $\R_p$ whose logarithmic derivative is $1$. Lemma~\ref{lem:magic} shows that $H(\sigma(x))$ is a good approximation for such an element; its logarithmic derivative is $1+x^{p-1}+x^{p^2-1}+\ldots\in \frac{1}{x}\F_p\osb x^p\csb$. By the additivity of the logarithmic derivative, we search for a factor eliminating the error term, i.e., an element of $\R_p$ whose logarithmic derivative is $-(x^{p-1}+x^{p^2-1}+\ldots)$. For this, choose $t=-g_1$, the primitive of this error series. Then $H(t)$ gives by Lemma~\ref{lem:magic} again a good approximation. Iterating this process we obtain exactly the infinite product defining $\exppt$. 
\end{rem} 
\begin{cor} \label{cor:algexp} 
For all $k\in \N$, the projections $\pi_k\left(\widetilde \exp_p\right)\in \Rp^{(k)}$ of $\widetilde \expp$ are algebraic over $\F_p(x, z_1,\ldots, z_k)$. Further, for each $\alpha\in \N^{(\N)}$ the series $\widetilde \exp_p\in \F_p\orb x, z_1,\ldots\crb$ has an algebraic Laurent series coefficient of $z^\alpha$ in $\F\orb x\crb$. The same holds true for $\exp_p$ and any algebraic multiple of it. 
\end{cor} 
\begin{proof} 
The series $g_k\in \F_p[z_1,\ldots, z_k]\osb x\csb$ are algebraic. Indeed, $g_k$ satisfies $g_k^p-g_k=g_{k-1}^pz_k$ and by induction and the transitivity of algebraicity, the claim follows. Moreover, $g_k\in z_k\F_p[z_1,\ldots, z_k]\osb x\csb$, so one sees that 
\[\prod_{i=0}^k H((-1)^ig_i)=\pi_k\left(\widetilde \exp_p\right),\] 
where the left-hand side is algebraic. Hence all the partial products are algebraic and approximate $\widetilde \exp_p$. The rest follows from Proposition~\ref{prop:monomially} and Lemma~\ref{lem:zcoeff}. 
\end{proof} 
 
Write $\widetilde{\exp}_p=\widetilde{e}_0+\widetilde{e}_1x+\widetilde{e}_2x^2+\ldots$ with $\widetilde{e}_i\in \F_p(z_1,z_2,\ldots)$. These coefficients $\widetilde{e}_i$ of $\widetilde{\exp}_p$ have the following remarkable property, alluded to at the beginning of this section and which uniquely determines the function $\exppt$ as a solution of $y'=y$. 
 
\begin{prop} \label{prop:period} 
For all $n\in \N$ we have $\widetilde{e}_{p^n-1}=(-1)^nx w_n'$. 
\end{prop} 
 
For the proof we need the following lemma describing certain coefficients of the polynomial $H(t)$: 
 
\begin{lem} \label{lem:coeffH} 
Write $H(t)=\sum_{i=0}^{p-1} a_i(t^p)t^i$. Then $a_{p-1}=-1$. 
\end{lem} 
\begin{proof} 
By Lemma~\ref{lem:magic} we have 
\[H(t)=(1-t^{p-1})H'(t).\] 
Comparing coefficients of powers of $t$ which are congruent to each other modulo $p$, one obtains the following recursion for the series $a_i$: 
\[a_i=(i+1)a_{i+1}-(i+2)a_{i+2}z^p \quad \text{for } i=0,\ldots, p-2 \quad \text{and} \quad a_{p-1}=-a_1.\] 
From this it follows that $a_{p-2}=-a_{p-1}$ and, inductively, that $a_i$ is divisible by $a_{p-1}$ for all $i$. Therefore $H(t)$ is divisible by $a_{p-1}$. Since $H(t)$ does not have a $p$-fold root, but $a_{p-1}\in \F_p[t^p]$, it follows that $a_{p-1}\in \F_p$ and $a_{p-1}=a_{p-1}(0)=-a_1(0)=-H'(0)=-1$. 
\end{proof} 
 
\begin{proof}[Proof of Proposition~\ref{prop:period}] 
Denote $h_i\coloneqq H((-1)^ig_i)$ and write $[x^k]f$ for the coefficient of $x^k$ in the Laurent series expansion of $f$. We show by induction 
\[\widetilde{e}_{p^n-1}=(-1)^m\prod_{i=0}^{m-1}w_i^{p-1}\cdot [x^{p^n-p^m}]\left(\prod_{i=m}^\infty h_i\right)\] 
for $m=0,\ldots, n$. For $m=0$ this is the definition of $\widetilde{e}_{p^n-1}.$ 
For the induction step we need to verify that for all $m$ we have 
\[ [x^{p^n-p^{m-1}}]\left(\prod_{i=m-1}^\infty h_i\right)=-w_{m-1}^{p-1}[x^{p^n-p^m}]\left(\prod_{i=m}^\infty h_i\right).\] 
Note that $g_i\in \F_p[z_1,z_2,\ldots]\osb x^{p^i}\csb $ and $g_i-w_ix^{p^i}\in \F_p[z_1, z_2,\ldots]\osb x^{p^{i+1}}\csb$. We see $\prod_{i=m}^\infty h_i\in \F_p[z_1,z_2,\ldots]\osb x^{p^m}\csb$ and 
\begin{equation} \label{eq:coeffrec} 
[x^{p^n-p^{m-1}}]\left(\prod_{i=m-1}^\infty h_i\right)=\sum_k[x^{kp^{m}-p^{m-1}}](h_{m-1})\cdot [x^{p^n-kp^m}]\left(\prod_{i=m}^\infty h_i\right). 
\end{equation} 
Now one easily checks by induction that \[h_{m-1}-H((-1)^{m-1}w_{m-1}x^{p^{m-1}})\in \F_p[z_1,z_2,\ldots]\osb x^{p^{m}}\csb\] and we obtain 
\[ [x^{kp^m-p^{m-1}}](h_{m-1})= [x^{kp^m-p^{m-1}}]\left(H((-1)^{m-1}w_{m-1}x^{p^{m-1}})\right),\] 
as the exponents considered are not multiples of $p^m$. 
Setting $s=(-1)^{m-1}w_{m-1}x^{p^{m-1}}$ in Lemma~\ref{lem:coeffH} we can further compute 
\[ [x^{kp^m-p^{m-1}}]\left(H((-1)^{m-1}w_{m-1}x^{p^{m-1}})\right)=\begin{cases} -((-1)^{m-1}w_{m-1})^{p-1} & \text{if } k=1\\ 0 & \text{otherwise,} \end{cases}\] 
which shows that the sum on the right-hand side of (\ref{eq:coeffrec}) only has one non-trivial summand, namely for $k=1$. This finishes the induction step. Now setting $m=n$ we obtain \[\widetilde{e}_{p^n-1}=(-1)^n\prod_{i=0}^{n-1}w_{i}^{p-1}=(-1)^nx  w_n'. \qedhere\] 
\end{proof} 
 
 
\section{The $p^k$-curvatures}\label{sec:pk-curvature} 
The theory developed in this section closely follows the results from \cite[\S 3.1.4, \S 3.2.1]{BCR23}. Let $L=\partial^n +a_{n-1}\partial^{n-1}+\ldots + a_1\partial +a_0 \in \F_p\orb x\crb [\partial]$ be a differential operator with rational functions, algebraic series or general power series as coefficients $a_i$. Rewrite the equation $Ly=0$ as a system of first order differential equations $Y'+AY=0$, where 
\[A=A_L\coloneqq \left(\begin{matrix} 
0 & -1 & 0 & \cdots & 0\\ 
0 & 0 & -1 & \cdots & 0\\ 
\vdots & \vdots & & \ddots & \vdots\\ 
0 & 0 & 0 &  & -1\\ 
a_0 & a_1 & a_2 &\ldots & a_{n-1} 
\end{matrix} \right)\] 
is the \textit{companion matrix} of $L$. The $p$-curvature of $L$ is defined as the map $(\partial+A)^p:\F_p\orb x \crb^n \to \F_p\orb x \crb^n$. It is an $\F_p\orb x \crb$-linear map (see Lemma \ref{lem:pcurvelin}) and plays an important role in the study of $Ly=0$. For example, its vanishing implies, for coefficients $a_i\in\F_p(x)$, the existence of a basis of $\F_p(x^p)$-linearly independent solutions in $\F_p(x)$ (see Cartier's Lemma, Proposition~\ref{prop:cartiersl}). More generally, the nilpotence of the $p$-curvature ensures various properties of $L$, for example that all its local exponents are contained in the prime field, see \cite[\S 5]{Hon81}. Defining the matrices $A_k\coloneqq(\partial+A)^k(I_n)$, one has $A_0=I_n, A_1=A$, with recursion formula $A_{k+1}=A_k'+AA_k$, and one checks that $A_p$ is the matrix of the $p$-curvature.\\ 
 
When considering solutions of differential equations in $\Rp$, i.e., with variables $z_i$ in their coefficients, the $p$-curvature itself do not control sufficiently the situation. At the same time, it is useful to quantize the order of nilpotence of the $p$-curvature. In analogy to the $p$-curvature we define, for any $k\geq 1$, the \textit{$p^k$-curvature} as the map  
\[(\partial+A_L)^{p^k}:\Rp^n\to \Rp^n.\] 
Thus, the $p^k$-curvature of $L$ vanishes for some $k$ if and only if the $p$-curvature of $L$ is nilpotent.\\ 
 
Recall that $\mathcal{R}_p^{(k)}\subseteq \Rp$ denotes $\F_p(z_1,\ldots, z_k)\orb x\crb$. 


\begin{lem} \label{lem:pcurvelin} 
The $p^k$-curvature of $L$ is an $\Rp^{(k-1)}$-linear map. Consequently, on $\left(\Rp^{(k-1)}\right)^n$, it is given by the evaluation $A_{p^k}\coloneqq(\partial +A)^{p^k}(I_n)\in \mathcal{M}_{n\times n}\left(\Rp^{(k-1)}\right)$. 
\end{lem} 
 
\begin{proof} 
By induction one easily shows for any $v\in \left(\Rp^{(k-1)}\right)^n$ and $f \in \Rp^{(k-1)}$ the equation 
\[(\partial + A)^{m}(fv)= \sum_{j=0}^m \binom{m}{j}\partial^{m-j}(f) (\partial + A)^j(v).\] 
In particular, for $m=p^k$, only two of the binomial coefficients do not vanish modulo $p$ and we obtain 
\[(\partial + A)^{p^k}(fv)=f(\partial + A)^{p^k}(v)+\partial^{p^k}(f)v=f(\partial + A)^{p^k}(v),\] 
as $\partial^{p^k}\left(\Rp^{(k-1)}\right)=0.$ 
\end{proof} 

For the rest of the section we will also consider some differential operators $L$ in which we allow $z_i$-variables in the coefficients of the operator. Without loss of generality assume $L$ to be monic. We will require that the coefficients of $L$ are in $\Spk$ for some $k$, or, equivalently, that the companion matrix $A_L$ has entries in  $\Spk$. For regular singular differential operators $L=\partial-a$ of order one with $a\in \F_p\osb x\csb$, this corresponds to a restriction of coefficients from $x^{-1} \F_p \osb x\csb$ to $\Sp^{(0)}=\F_p\osb x\csb$. This can be done without loss of generality.  
Indeed, consider the equation 
\begin{equation*} 
y'+ ay=0. 
\end{equation*} 
for $a\in \F_p\orb x\crb$ and assume that it is regular singular, i.e., $a$ has a pole of order at most one at $0$. Write $L=\partial+a$ for the corresponding differential operator and assume that its initial form $L_0$ is an element of $\F_p[x][\partial]$, i.e., $a=\frac{\rho}{x}+\widetilde a$, where $\rho \in \F_p$ is its local exponent and $\widetilde a \in \F_p\osb x\csb$.  Replacing $y$ by $x^{-\rho}y$, the differential equation changes to 
\[ y'+\left(a-\frac{\rho}{x}\right) y=0,\] 
whose local exponent is $0$ and which has, equivalently, no singularity at $0$.\\ 
 
Recall that the elements $\widetilde{e}_i\in \F_p[z_1,z_2,\ldots]$ were defined as the coefficients of the solution  $\widetilde{\exp}_p=\widetilde{e}_0+\widetilde{e}_1x+\widetilde{e}_2x^2+\ldots$ of the exponential differential equation $y'=y$ in Section~\ref{sec:expp}. In particular, this means that $(\widetilde{e}_ix^i)'=\widetilde{e}_{i-1}x^{i-1}.$ 


\begin{lem} \label{lem:solution} 
Let $Y'+AY=0$ be a matrix differential equation without singularity at $x=0$, i.e., $A\in \mathcal{M}_{n\times n}(\F_p\osb x\csb)$, or, more generally, $A\in\mathcal{M}_{n\times n}(\Sp)$. Set $A_i=(\partial+A)^i(I_n)$. Then the matrix 
\begin{equation} \label{eq:solTaylor} 
Y\coloneqq \sum_{i=0}^\infty (-1)^i\widetilde{e}_ix^iA_i 
\end{equation} 
is an element of $\mathcal{M}_{n\times n}(\Sp)$ and a fundamental matrix of solution of $Y'+AY=0$. 
\end{lem} 
 
\begin{rem} 
The choice of a particular exponential function does not matter here. More precisely, one might replace $\widetilde {e_i}$ by the coefficients of the power series expansion of any other exponential function, i.e., any solution of $y'=y$ in $\Sp$, e.g., the coefficients $e_i$ of the xeric solution $\exp_p$. 
\end{rem} 
 
\begin{proof}[Proof of Lemma~\ref{lem:solution}] 
The first assertion is easy to see: If $A\in \Sp$,  then $A_i\in \Sp$ for all $i$, as $\Sp$ is a differentially closed ring. Thus $\ord_x \left(\widetilde{e}_ix^iA_i\right)\geq i,$ so the sum is a well-defined element of $\mathcal{M}_{n\times n}(\Sp)$.\\ 
 
For the second part we compute using the product rule 
\begin{align*} 
Y'&= \sum_{i=0}^\infty (-1)^i\widetilde{e}_ix^iA_i'+ \sum_{i=1}^\infty (-1)^i\widetilde{e}_{i-1}x^{i-1}A_i\\ 
&=\sum_{i=0}^\infty (-1)^i\widetilde{e}_ix^i(A_i'-A_{i+1})\\ 
&=\sum_{i=0}^\infty (-1)^i\widetilde{e}_ix^i(-AA_i) =-AY. \qedhere 
\end{align*} 
\end{proof} 
 
Now setting $x=0$ in $Y$, we obtain $\det Y(0)= \det I_n=1$, so $Y(0)\in \GL_n(\F_p)$ and $Y\in \GL_n(\Rp)$. 
 
\begin{rem} 
This Lemma gives another verification of the fact that any non-singular differential equation has a full basis of solutions in $\R_p$. 
\end{rem} 
 
The following proposition generalizes Cartier's Lemma: It relates the vanishing of the $p$-curvature to the existence of solutions of differential equations in positive characteristic $p$, see \cite{Kat72}, respectively, Thm.~3.18 in \cite{BCR23}. \\
 
 
\newcommand\Tp{\mathcal{T}_p} 
 
\begin{prop}[Extension of Cartier's Lemma]\label{prop:cartiersl} Let $L$ be a monic differential operator in $\Spk[\partial]$ and let $\partial+A$ be the corresponding first order matrix differential operator with $p^{k+1}$-curvature $A_{p^{k+1}}$. 
 
Denote by $\Tp^{(k)}$ one of the fields $\F_p(x, z_1,\ldots, z_k)$, $\R _{p,\text{alg}}^{(k)}$, or $\Rp^{(k)}$, where the second denotes the subfield of $\Rp^{(k)}$ of elements algebraic over $\F_p(x, z_1,\ldots, z_k)$. Assume that $L$ has coefficients in $\Tp^{(k)}$. Then the following are equivalent:
 
\begin{enumerate}[(i)] 
\item  The equation $Ly=0$ has a full basis of solutions in $\Tp^{(k)}$. 
\item  The $p^{k+1}$-curvature of $L$ vanishes, $A_{p^{k+1}}=0$. 
\item  The operator $\partial^{p^{k+1}}$ is right-divisible by $L$ in $\Tp^{(k)}[\partial]$. 
\end{enumerate}

\end{prop} 
 
\begin{proof} 
The proof for all three assertions works analogously, we do it for $\Tp^{(k)}=\F_p(x, z_1,\ldots, z_k)$. Assume (i) holds. There then exists a fundamental matrix $Y\in \GL_n(\F_p(z_1,\ldots, z_k,x))$ of solutions of $(\partial+A)Y=0$.  Then 
\[A_{p^{k+1}}= (\partial+A)^{p^{k+1}}(I_n)=Y^{-1}L^{p^{k+1}}(Y)=0,\] 
where we have used the linearity of the $p^{k+1}$-curvature on $\Rpk$, see Lemma~\ref{lem:pcurvelin}. This shows (ii). For the converse implication consider the fundamental matrix of solutions $Y$ of $(\partial+A_L)Y=0$ given in Lemma~\ref{lem:solution}, Equation~\eqref{eq:solTaylor}. As $A_{p^{k+1}}=L^{p^{k+1}}(I_n)=0$ this is a finite sum of rational functions in the variables $x, z_1,\ldots, z_k$, which proves (i).\\ 
 
For the equivalence of (i) and (iii) we use that $\F_p(x, z_1,\ldots, z_k)[\partial]$ is a (left- and right-) Euclidean ring. In fact, the skew-polynomial ring over any (skew-) field is Euclidean, as observed by Ore \cite{Ore33}. Thus we may write $\partial^{p^{k+1}}=QL+R$ for some differential operators $R, Q$, with $\ord\, R<\ord\, L =n$. If $y\in \F_p(x, z_1,\ldots, z_k)$ is a solution of $Ly=0$, then $Ry=\partial^{p^{k+1}}y-QLy=0$, and consequently each solution of $Ly=0$ is also a solution of $Ry=0$. Because the solutions of $L$ form an $n$-dimensional vector space over the constants by assumption and $R$ is of order smaller than $n$, this means $R=0$. Thus (iii) follows. 
 
Conversely, assume that $\partial^{p^{k+1}}=QL$. The kernel of $\partial^{p^{k+1}}$ is $\F_p(x, z_1,\ldots, z_k)$ and thus a $p^{k+1}$-dimensional $\F_p(x^p, z_1^p,\ldots, z_k^p)$-vector space. The $\F_p(x^p, z_1^p,\ldots, z_k^p)$-dimensions of the kernels of $Q$, respectively, $L$ in $\F_p(x, z_1,\ldots, z_k)$ are at most $p^{k+1}-n$, respectively, $n$, thus equality must hold. 
\end{proof} 
 
We conclude the section with an explicit formula for the $p^k$-curvature of order one equations. It is a generalization of the formula $a_p=a^p+a^{(p-1)}$ for the $p$-curvature for first order equations $y'+ay=0$ with rational function coefficients \cite[Thm.~3.12]{BCR23}. It seems to have no obvious extension to higher order differential equations. 


\begin{prop}\label{prop:pcurveformula} 
Let $L=(\partial+a)$ be a differential operator with $a\in \Rp$. The $p^k$-curvature $a_{p^k}\coloneqq L^{p^k}(1)$ is given by: 
\begin{equation*} 
a_{p^k}=\sum_{i=0}^k\left(a^{(p^i-1)}\right)^{p^{k-i}}=\left(a_{p^{k-1}}\right)^p+a^{(p^k-1)}. 
\end{equation*} 
\end{prop} 
 
\begin{proof} 
Write $a_m$ for $(\partial+a)^m(1)$. First note that $a_m$ can be written as 
\begin{equation}\label{eq:am} 
a_m=\sum_{\alpha \in \mathcal{A}_m} \lambda_\alpha \prod_{j=0}^{m-1} \left(a^{(j)}\right)^{\alpha_j}, 
\end{equation} 
where \[\mathcal{A}_m\coloneqq \{\alpha=(\alpha_0,\ldots, \alpha_{m-1}) \in \N^m: \sum_{j=0}^{m-1}\alpha_j(j+1)=m\}.\] 
Indeed, $(\partial+a)(1)=a$ and inductively, for each summand in (\ref{eq:am}) both, multiplication by $a$ and differentiation, give a monomial with exponents in $\mathcal{A}_{m+1}$. 
 
Next, we show that each of the coefficients $\lambda_\alpha\in\F_p$ for $\alpha\in\mathcal{A}_m$ in the expansion (\ref{eq:am}) of $a_m$ is given by 
\begin{align} \notag 
\lambda_\alpha &= \frac{m!}{\prod_{j=0}^{m-1}\alpha_j! ((j+1)!)^{\alpha_j}}\\ 
&=\binom{m}{\alpha_0, 2\alpha_1, \ldots, m\alpha_{m-1}}\prod_{j=0}^{m-1} \binom{\alpha_j(j+1)}{j+1,\ldots, j+1}\frac{1}{\alpha_j!}\label{eq:coeff}. 
\end{align} 
The product on the right-hand side of this equation is an integer; thus its residue modulo $p$ defines an element in $\F_p$, equal to $\lambda_\alpha$. Again, we proceed by induction. For $m=0$, $\mathcal{A}_0=\emptyset$ and $\lambda_0=1$. Denote by $\epsilon_k$ for $0\leq k \leq m$ the element $(0,\ldots, 0, 1, 0, \ldots, 0)\in \N^{m+1}$, where the entry $1$ is in the $k+1$-st position. We embed $\N^{m}$ in $\N^{m+1}$ by $\N^m\cong \N^m\times \{0\}\subseteq \N^{m+1}$.  Then for $\alpha\in \mathcal{A}_{m+1}$ we get 
\begin{align*} 
\lambda_\alpha&=\lambda_{\alpha-\epsilon_0}+\sum_{j=0}^{m-1}(\alpha_{j}+1)\lambda_{\alpha+\epsilon_j-\epsilon_{j+1}}\\ 
&=\frac{(m+1)!}{\prod_{j=0}^m\alpha_j!((j+1)!)^{\alpha_j}}\left(
\frac{\alpha_0\cdot 1!}{m+1}+\sum_{j=0}^m \frac{(j+2) \alpha_{j+1}}{m+1} 
\right)=\frac{(m+1)!}{\prod_{j=0}^m\alpha_j!((j+1)!)^{\alpha_j}} 
\end{align*} 
using the induction hypothesis for $\alpha+\epsilon_j-\epsilon_{j+1}\in \mathcal{A}_m$.\\ 
 
Now we show that for $m=p^k$ only a small portion of the coefficients $\lambda_\alpha$ are non-zero, namely only if $\alpha=p^{k-\ell}\epsilon_{p^\ell-1}$ for $\ell=0,1,\ldots, k$. In these cases $\lambda_\alpha=1$. 
 
By Lucas' Theorem applied to the left hand side multinomial coefficient in \eqref{eq:coeff}, it follows that $\lambda_\alpha=0$ except $p^k=m=\alpha_{j_0}(j_0+1)$ for some $j_0$. Consequently $j_0=p^\ell-1$ and $\alpha_{j_0}=p^{k-\ell}$ for some $\ell$. This means $\alpha=p^{k-\ell}\epsilon_{p^\ell-1}$. We compute, splitting the multinomial coefficient into a product of binomial coefficients, accounting for one factor of $p^{k-\ell}!$ in each of these binomials and using Lucas' Theorem: 
\[\lambda_\alpha=1\cdot \frac{1}{p^{k-\ell}!}\binom{p^k}{p^\ell,\ldots, p^\ell}=\prod_{j=0}^{p^{k-\ell}}\binom{jp^\ell-1}{p^\ell-1}=\prod_{j=0}^{p^{k-\ell}}\binom{j-1}{0}=1.\] 
This finishes the proof. 
\end{proof} 
 
If $a\in \F_p\osb x\csb$, then $a_{p^k}=a_p^{p^{k-1}},$ i.e., the evaluation of $L^{p^k}$ at $1$ is just the $p^{k-1}$-st power of the evaluation of $L^p$ at $0$. Therefore it vanishes if and only if the $p$-curvature does. However, if $a\in \F_p(z_1,\ldots, z_k)\orb x\crb$ depends on finitely many $z_i$-variables, the associated $p^j$-curvatures are not powers of each other if $j\leq k$, whereas, for $j>k$, one still has $a_{p^j}=a_{p^{k+1}}^{j-k}$. If $a$ is an arbitrary element in $\R_p$, there need not be any such relation between the $p^j$-curvatures. 
  
 
\section{Product Formulas and Algebraicity for Solutions of Equations of Order 1}\label{sec:product} 
 
The goal of this section is to generalize the product formula for $\exp_p$ developed in Section~\ref{sec:expp} to solutions of arbitrary first order differential equations. 
\\ 
 
Recall from the previous section that for regular singular first order operators $\partial-a$ we can restrict to the study of non-singular operators by replacing $y$ by $x^{-\rho}y$, where $\rho$ is the local exponent of the equation and can assume that  $a\in \Sp^{(0)}=\F_p\osb x\csb$.  We defined $w_k\coloneqq z_1^{p^{k-1}}z_2^{p^{k-2}}\cdots z_k^{1}$ and  $\mathcal{S}_p$ as the completion of $\F_p[x, x^pw_1,x^{p^2}w_2,x^{p^3}w_3,\ldots] $ in $\Rp$ with respect to the $x$-adic topology, as well as $\mathcal{S}_p^{(k)} \coloneqq \mathcal{S}_p\cap \Rpk$. Further recall the projection maps $\pi_k:\Sp\to \Spk$, given by setting $z_{k+1}, z_{k+2},\ldots$ equal to $0$. 


\begin{thm}\label{thm:prod} 
Let $L=\partial+a$ be a first order linear differential operator with $a$ a rational function in $\F_p(x)$ or an algebraic series in $\F_p\orb x\crb$. Assume that $L$ has a regular singularity at $0$ and local exponent $\rho=0$. Then, for every $k\in \N$, there exists a series $h_k\in 1+x^{p^k}w_k\Spk$, algebraic over $\F_p(x, z_1,z_2,\ldots, z_k)$, such that 
$y=\prod_{k=0}^\infty h_k$ is a solution of $Ly=0$. In particular, $\pi_i(h)=\prod_{k=0}^ih_k$ is algebraic over $\F_p(x,z_1,\ldots, z_i)$ for all $i$. 
\end{thm} 
 
The series $h_k$ will be explicitly constructed in the course of the proof. Combining Proposition~\ref{prop:monomially} and Lemma~\ref{lem:zcoeff}, one gets the following immediate consequences. 
 
\begin{cor} \label{cor:xericalgebraic}
Let $L=\partial+a$ be a first order differential operator with rational or algebraic power series coefficient $a\in \F_p\orb x\crb$ and local exponent $\rho \in \F_p$. Then its \monominduced solution has algebraic projections. 
\end{cor} 
 
\begin{cor} 
Let $y$ be the solution $h$ of $Ly=0$ defined in Theorem~\ref{thm:prod} or the \monominduced solution of this equation. Then, for any $\alpha\in \N^{(\N)}$, the coefficient $y_\alpha\in \F_p\orb x\crb$ of $z^\alpha$ in $y$ is algebraic over $\F_p(x)$. In particular, the initial series $y\vert_{z_1=z_2=\cdots=0}$ of $y$ is algebraic. 
\end{cor} 
 
For the proof of Theorem~\ref{thm:prod} we need to deform $L$ to an operator $L_{\leq i}=L-V_i$ with vanishing $p^{i+1}$-curvature, and such that the solutions of $L$, $L_{\leq i}$ and $L_{>i}=\partial +V_i$ are closely related to each other. To do so, we need two auxiliary results.\\


\begin{lem}\label{lem:diffprod} 
Let $a, b\in \R_p$ and let $s\in \Rp\setminus \{0\}$ be a solution of $(\partial +a)y=0$. 
Then 
$t$ is a solution of $(\partial + b)y=0$ 
if and only if 
$st$ is a solution of $(\partial +a + b)y=0$. 
\end{lem} 
\begin{proof} 
It is clear since 
$(\partial +a + b)(st)=(s'+as)t+s(t'+bt)=s(\partial +b)t$. 
\end{proof} 


\begin{lem}\label{lem:proj1} 
Let $L=\partial +  v^p (x^{p^{i+1}}w_{i+1})'$ be a 
differential operator with $v\in \Spi$. 
Then, 
there exists a solution $y_0\in\Spi$ of $Ly=0$ whose 
$i$-th projection is $1$, i.e., $\pi_i(y_0)=1$. 
\end{lem} 
\begin{proof} 
If $v=0$, we only have to take $y_0=1$.  Assume $v\neq0$. 
Define an $\mathbb{F}_p$-algebra endomorphism 
$\varphi\colon\Sp\to\Sp$ 
and a derivation $D$ on $\varphi(\Sp)$ by 
$$\varphi(x^{p^{k-1}}w_{k-1})\coloneqq v^{p^{k}}x^{p^{i+k}}w_{i+k}, 
\quad 
D 
\coloneqq(\varphi(x)')^{-1}\partial. 
$$ 
Since 
$\partial (x^{p^{k-1}}w_{k-1}) 
=\prod_{j=0}^{k-2}(x^{p^{j}}w_{j})^{p-1}$ 
and 
$$D (\varphi(x^{p^{k-1}}w_{k-1})) 
=\dfrac{v^{p^k}(x^{p^{i+k}}w_{i+k})'}{v^p(x^{p^{i+1}}w_{i+1})'} 
=v^{p^k-p}\prod_{j=i+1}^{k+i-1}(x^{p^{j}}w_{j})^{p-1} 
=\prod_{j=0}^{k-2}\varphi(x^{p^{j}}w_{j})^{p-1} 
,$$ 
$\varphi$ induces an isomorphism 
$(\Sp,\partial)\cong(\varphi(\Sp),D)$ of differential algebras. 
Set $y_1=\exppt(-x)= \sum_{j=0}^\infty \widetilde{e}_j(-x)^j$ 
as in Lemma \ref{lem:solution}.
Since $(\partial+1)y_1=0$, the image $y_0=\varphi(y_1)$ satisfies 
$Ly_0=\varphi(x)'(D+1)y_0=0$.  Since ${y_1}|_{x=0}=1$, 
we have $y_0|_{w_{i+1}=0}=1$, and therefore, $\pi_{i}(y_0)=1$. 
\end{proof}


\begin{lem} \label{lem:Li} 
Let $L= \partial+a$ be a differential operator with algebraic coefficient $a\in \Spi$, for some $i\geq 1$. There exists an element 
$V_i=v^p (x^{p^{i+1}}w_{i+1})'$ with 
$v\in \Spi$, algebraic over $\F_p(x, z_1,\ldots, z_i)$, such that the operators 
$L_{\leq i}=L-V_i$ and $L_{>i}=\partial +V_i$ satisfy the following properties. 
 
\begin{enumerate}[(i)] 
\item $L_{\leq i}$ has zero $p^{i+1}$-curvature. 
 
\item $L_{>i}y=0$ has a solution $y\in \Sp$ with projection $\pi_i(y)=1$. 
 
\item For solutions $s$ of $L_{\leq i}y=0$ and 
$t$ of $L_{>i}y=0$, the product $y=st$ is a solution of $Ly=0$. 
\end{enumerate} 
\end{lem}

\begin{proof} 
 
Take an arbitrary element $v\in \Sp^{(i)}$.  We set 
$V_i=v^p (x^{p^{i+1}}w_{i+1})'$ and 
$L_{\leq i}=L-V_i=\partial +a-V_i$. 
 
By Lemma \ref{lem:derivationrules}, we have 
$(a^{(p^{i+1}-1)})'=a^{p^{i+1}}=0$. 
It follows that 
$a^{(p^{i+1}-1)}\in\mathcal{C}_p\cap\Spi$, 
and hence there exists $\widetilde{a}\in\Spi$ 
such that $a^{(p^{i+1}-1)}={\widetilde{a}}^p$. 
 
By Proposition~\ref{prop:pcurveformula} the $p^{i+1}$-curvature of $L_{\leq i}$ is given by 
\begin{align*} 
L_{\leq i}^{p^{i+1}} (1) 
&=L_{\leq i}^{p^{i}} (1)^p + (a-V_i)^{(p^{i+1}-1)} 
=L_{\leq i}^{p^{i}} (1)^p +(a)^{(p^{i+1}-1)}-v^p (x^{p^{i+1}}w_{i+1})^{(p^{i+1})}\\ 
&=L_{\leq i}^{p^{i}} (1)^p + {\widetilde a}^p - (-1)^{i+1} v^p 
=\left\{
L_{\leq i}^{p^{i}} (1) + {\widetilde a} - (-1)^{i+1} v 
\right\}^p. 
\end{align*} 
The vanishing of the $p^{i+1}$-curvature of $L_{\leq i}$ is thus equivalent to 
\[(-1)^{i+1}\left(L_{\leq i}^{p^i}(1)+\widetilde a\right) = v.\] 
The left hand side of this equation can be expanded as a polynomial in $v^p$ with algebraic coefficients in $\Spi$. Thus we can invoke the implicit function theorem to find an algebraic solution $v$. This shows (i). 
By definition of $V_i$, $L_{\leq i}$ and $L_{>i}$, assertions (ii) and (iii) are 
direct consequences of Lemmata \ref{lem:diffprod} and \ref{lem:proj1}. 
\end{proof} 
 
 
\begin{proof}[Proof of Theorem~\ref{thm:prod}.] 
For $i\in \N$, we will construct inductively algebraic elements $h_i\in \Spi$  
such that $\pi_{i-1}(h_i)=1$ and such that 
$\prod_{j=0}^\infty h_j$ 
converges to a solution of $Ly=0$ 
in the $x$-adic topology. 
We apply Lemma \ref{lem:Li} to define for $i\in \N$
operators $N_i$ 
inductively by the formula 
$$ 
N_{0}=L,\quad 
N_{i+1}=(N_{i})_{> i}=\partial+V_i,$$
with $V_i$ as in Lemma~\ref{lem:Li}. 
Then \[(N_i)_{\leq i}=N_{i}-V_i=\partial+V_{i-1}-V_i\] for $i\geq 1$ and $(N_0)_{\leq 0}=\partial+a-V_0$. We define solutions 
$s_i\in\Spi$ 
of $(N_{i})_{\leq i}y=0$ by the formula 
$$ 
s_i\coloneqq \sum_{j=0}^{p^{i+1}-1}(-1)^j\widetilde{e}_j 
(N_{i})_{\leq i}^j(1)x^j. 
$$ 
By Lemma \ref{lem:Li} (i) the $p^{i+1}$ curvature of $(N_{i})_{\leq i}$ vanishes, and thus $s_i$ is indeed the solution of $(N_{i})_{\leq i}y=0$ given in Lemma~ \ref{lem:solution}.
By Lemma \ref{lem:Li} (ii), 
we can also take solutions 
$t_i\in\Spi$ of $N_{i+1}y=0$ for $i\in \N$
satisfying $\pi_i(t_i)=1$. 
 
\medskip 
 
Since $t_{i-1}$ and $s_it_{i}$ are both solutions of 
$N_iy=0$, they coincide up to ${\mathcal{C}_p}$-multiplication. 
As $\pi_{i-1}(t_{i-1})=\pi_{i-1}(t_i)=1$, we conclude that 
$\pi_{i-1}(s_i)\in{\mathcal C}_p$.  We observe 
that $\pi_{i-1}(s_i)$ is a unit in $\Spi$ and algebraic, and
we define $h_i$ by 
$$h_i=\pi_{i-1}(s_i)^{-1}s_i.$$ 
Then it is clear that 
$h_i\in \Spi$ is a solution of $(N_{i})_{\leq i}y=0$, 
algebraic over $\F_p(x,z_1,\ldots, z_i)$ and $\pi_{i-1}(h_i)=1$, i.e., $h_i\in 1+x^{p^i}w_i\Spi$. \\
 
The last thing to show is that $\prod_{j=0}^\infty h_j$ converges to a 
solution of $Ly=0$. For $i\geq 0$, set 
$$b_i=\prod_{j=0}^ih_j\in\Spi.$$
As $h_i$ is a solution of $(N_i)_{\leq i}y=0$, induction in combination with Lemma~\ref{lem:diffprod} shows that $b_i$ is a solution of $(L-V_i)y=0$. 

Recall that if $f\in\Sp$ satisfies $\pi_\ell(f)=0$, then 
$\ord_x(f)\geq p^\ell$ holds.  Actually, 
$f\in\Sp$ and $\pi_\ell(f)=0$ imply that 
$f$ belongs to the $\Sp$-module generated by the elements $x^{p^k}w_k$ 
with $k\geq \ell$, and hence $\ord_x(f)\geq {p^\ell}$. 
 
By this observation, we have 
$b_i= b_\ell\prod_{\ell<j\leq i}h_j\equiv b_\ell \mod x^{p^\ell}$ for $i\geq \ell$. 
Thus $b\coloneqq \lim_{i\to\infty}b_i=\prod_{j=0}^\infty h_j$ converges in the 
$x$-adic topology. 
 
As $(\partial+V_i)t_i=0$ and $(L-V_i)b_i=0$, we see again by Lemma \ref{lem:diffprod} that $t_ib_i$ is a solution of 
$Ly=0$ for any $i$. 
Since $t_i\in\Sp$ and $\pi_i(t_i)=1$, we have 
$\ord_x(b_i-t_ib_i)\geq p^i$.  
It follows that 
$\ord_x(Lb_i)=\ord_x(L(t_ib_i)+L(b_i-t_ib_i))\geq p^i-1$ and hence $Lb=\lim_{i\to \infty}Lb_i=0$.
\end{proof} 


\begin{rem} 
We can replace $a$ by an algebraic element of $\Spk$ in the theorem above. For the proof one sets 
\[h_i\coloneqq \pi_{i-1}\big(s_k\big)^{-1}\pi_i\big(s_k\big)\] for $1\leq i\leq k$ and $h_0=\pi_0(s_k)$. Since $s_k$ is algebraic, so are $h_0,\ldots, h_k$, and 
\[\pi_{i-1}\big(h_i\big)=\pi_{i-1}\big(\pi_{i-1}\big(s_k\big)^{-1}\pi_i\big(s_k\big)\big)=1. \] 
For $h_i$ with $i>k$ one proceeds as in the proof of the theorem. 
\end{rem} 
 
\begin{ex} 
In the case of the exponential differential equation $y'=y$ we have $a=-1$. In this case the equation for $v_0$ reads according to Lemma~\ref{lem:Li} 
\[-1-v_0^px^{p-1}+v_0=0,\] 
which has the solution $v_0=x^{-1}\sigma(x)$, where $\sigma(x)=\sum_{k=0}^\infty x^{p^k}$. The differential equation $L_{\leq 0}$ then reads $xy'=\sigma(x)y$, or equivalently: 
\[\frac{y'}{y}=\frac{\sigma(x)}{x}.\] By Lemma~\ref{lem:magic} we have 
\[\frac{H(\sigma(x))'}{H(\sigma(x))}=\frac{\sigma(x)}{x}, \] 
which shows that 
\[H(\sigma(x))=\prod_{k=1}^{p-1}\left(1-\frac{1}{k}\sigma(x)\right)^k\eqqcolon h_0 \] 
solves the equation $L_{\leq 0}y=0$. So we recover the beginning of the infinite product defining $\exppt$. 
\end{ex} 


\section{Trigonometric functions} \label{sec:trig} 
Having investigated the exponential function in positive characteristic one cannot resist to look also at the sine and cosine function, i.e., at the solutions of the second order differential equation 
 
\[y''+y=0.\] 
 
The local exponents at $0$ are $0$ and $1$. The equation has a $2$-dimensional solution space over the field of constants $\Cp =\F_p(z_1^p, z_2^p,\ldots)\orb x^p \crb$. Here are, for $p=3$, the two \monominduced solutions, which we call $\sin_p$ and $\cos_p$: 
\[\sin_3(x) = x+z_1 x^3+z_1 x^5+(z_1^2+z_1) x^7+z_1^3 z_2 x^9+(z_1^3 z_2+2 z_1) x^{11}+\ldots\] 
 
\[\cos_3(x) =1+x^2+2 z_1 x^4+(z_1^2+2 z_1) x^6+(z_1^2+2 z_1+2) x^8+(2 z_1^3 z_2+z_1) x^{10}+\ldots\] 
 
In this situation, it is tempting to expect again an algebraic relation of the form $\sin_p^2+\cos_p^2 =1$ as in the characteristic zero case. This can easily be disproved, and it is also not clear a priori how $\sin_p$ and $\cos_p$ relate to $\exp_p$. To explore these questions, expand 
\[\exp_p=e_0+e_1x+ e_2x^2+\ldots\] 
with $e_i\in \F_p[z]$ and split $\exp_p$ into 
\[\eve(\exp_p)=e_0+ e_2x^2+e_4x^4+\ldots,\] 
\[\odd(\exp_p)=e_1x+ e_3x^3+e_5x^5+\ldots\] 
as series of even and odd degrees. Clearly, both series are solutions of $y''-y=0$ since $(e_ix^i)''=e_{i-2}x^{i-2}$. They are, however, not \monominduced. Let us denote by $\sinh_p$ and $\cosh_p$ the \monominduced solutions of $y''-y=0$. Further, for $\ch(K)>2$, it is immediate that 
\[\eve(\exp_p)=\frac{1}{2}\left(\exp_p(z,x)+\exp_p(z,-x)\right),\] 
\[\odd(\exp_p)=\frac{1}{2}\left(\exp_p(z,x)-\exp_p(z,-x)\right).\] 
This proves by Corollary~\ref{cor:algexp} that both $\eve(\exp_p)$ and $\odd(\exp_p)$ have algebraic projections: they play the role of the classical hyperbolic sine and cosine functions $\sinh$ and $\cosh$ in  characteristic $p$. By Proposition~\ref{prop:monomially}, the corresponding \monominduced solutions, $\sinh_p$ and $\cosh_p$, also have algebraic projections.\\ 
 
The same argument applies for the equation $y''+y=0$ and $\ch(K)>2$. The two series 
\[\frac{1}{2}\left(\exp_p(z,ix)+\exp_p(z,-ix)\right),\] 
\[\frac{1}{2}\left(\exp_p(z,ix)-\exp_p(z,-ix)\right),\] 
where $i\in\overline\F_p$ is a square root of $-1$, form a basis of solutions. This proves: 
 
\begin{prop}\label{prop:trigalg} 
The projections of $\cosh_p,\ \sinh_p,\ \cos_p,\ \sin_p$ are all algebraic. 
\end{prop} 
 
The next observation is somewhat more surprising.

\begin{prop} \label{prop:trigrel} 
Let $\sinh_p$ and $\cosh_p$ denote the \monominduced  solutions of $y''-y=0$ with respect to the local exponents $\rho_1=1$ and $\rho_2 =1$. Then the following identity holds, 
\[\exp_p=\cosh_p+\frac{1}{1-\sigma^p}\sinh_p,\] 
where $\sigma(x)=x+x^p+x^{p^2}+\ldots$ 
\end{prop} 
 
\begin{rem} In this formula, there is an asymmetry between $\sinh_p$ and $\cosh_p$. On the other hand, by definition the symmetric formula $\expp= \eve(\expp)+\odd(\expp)$ holds. 
\end{rem} 
 
\begin{proof} The functions $\sinh_p$ and $\cosh_p$, as \monominduced solutions, are uniquely determined as the solutions of $y''=y$ with  $\langle \sinh_p\rangle_{0,0}=0, \langle \sinh_p\rangle_{1,0}=1$ respectively $\langle \cosh_p\rangle_{0,0}=1, \langle \cosh_p\rangle_{1,0}=0$.\\ 
 
Write $\eve_p$ and $\odd_p$ for $\eve(\exp_p)$ and $\odd(\exp_p)$. Note that $\langle \exp_p\rangle_{0,0}=1$ and consequently $\langle \odd_p\rangle_{0,0}=0$ and $\langle \eve_p\rangle_{0,0}=1$. A short computation shows that 
\[\sinh_p=\odd_p\cdot \frac{x}{\langle \odd_p\rangle_{1,0}}\] 
and 
\[\cosh_p=\eve_p - \sinh_p \frac{\langle  \eve_p\rangle_{1,0}}{x}=\eve_p-\odd_p\frac{\langle  \eve_p\rangle_{1,0}}{\langle \odd_p\rangle_{1,0}}, \] 
since for example 
\[ \left\langle \odd_p\cdot \frac{x}{\langle \odd_p\rangle_{1,0}}\right \rangle_{1,0}=\langle \odd_p \rangle_{1,0}\cdot \frac{x}{\langle \odd_p\rangle_{1,0}}=x \] 
holds. As $\exp_p=\odd_p+\eve_p$ we obtain 
\[\exp_p= \cosh_p + \frac{1}{x}\sinh_p \cdot \langle \odd_p+\eve_p\rangle_{1,0}.\] 
We set 
\[K\coloneqq \frac{1}{x}\langle \odd_p+\eve_p\rangle_{1,0}=\frac{1}{x}\langle \exp_p\rangle_{1,0},\] 
and we are left to show that $\displaystyle K=\frac{1}{1-\sigma^p}.$\\ 
 
Recall the infinite product decomposition of the solution $\exppt=h_0\cdot h_1\cdot h_2\cdots$ of $y'=y$, where $h_i=H((-1)^i g_i)$, the $g_i$ are defined recursively and $H$ is a polynomial satisfying $H(s)=(1-s^{p-1})H'(s)$, see Lemma~\ref{lem:magic}. Then $\exppt = \expp \langle \exppt \rangle_{0,0}$ and substituting in the definition of $K$ we obtain 
\[xK \langle \exppt \rangle_{0,0}= \langle \exppt \rangle_{1,0}.\] 
We have the equality 
\[ \langle \exppt\rangle_{0,0} = \langle h_0\rangle _{0,0}\cdot \langle h_0^{-1}\exppt\rangle_{0,0}.\] 
 Indeed, $h_0^{-1}\exppt=h_1\cdot h_2\cdots$ is a series in $x^p$ with coefficients in $\F_p(z_1, z_2,\ldots)$ and such that $\pi_0\big(h_0^{-1}\exppt\big)=1$. Moreover, $h_0$ is a series in $x$ only. Therefore a monomial in $\langle \exppt\rangle_{0,0}$ can be written uniquely as a product of a monomial in $\langle h_0\rangle _{0,0}$ and a monomial in $\langle h_0^{-1}\exppt\rangle_{0,0}$.\\ 
 
Similarly we obtain 
\[ \langle \exppt\rangle_{1,0} = \langle h_0\rangle _{1,0}\cdot \langle h_0^{-1}\exppt\rangle_{0,0}\] 
and consequently 
\[xK \langle \exppt \rangle_{0,0}= \langle h_0\rangle _{1,0}\cdot \langle h_0^{-1}\exppt\rangle_{0,0}= \langle h_0\rangle _{1,0}\cdot\langle \exppt\rangle_{0,0}\cdot \langle h_0\rangle _{0,0}^{-1}.\] 
Using that $\langle h_0\rangle_{1,0}=x\langle h_0'\rangle_{0,0}$ we get $K=\langle h_0'\rangle_{0,0}\langle h_0\rangle_{0,0}^{-1}.$ Recall that $h_0=H(\sigma)$ and thus, by Lemma~\ref{lem:magic} and the identity $\sigma-\sigma^p=x$ we have 
\[\frac{h_0'}{h_0}=\frac{\sigma}{\sigma-\sigma^{p}}=\frac{\sigma}{x}.\] 
Moreover, 
\[\langle h_0'\rangle_{0,0}=\left \langle \frac{h_0 \sigma}{x}\right \rangle_{0,0}= \langle h_0 \rangle_{0,0}+ \sigma^p\left \langle \frac{h_0}{x}\right \rangle_{0,0}= \langle h_0\rangle_{0,0}+\sigma^p\langle h_0'\rangle_{0,0}\] 
and thus $\displaystyle K=\frac{1}{1-\sigma^p}$. 
\end{proof} 
 
The considerations in this chapter motivate to investigate the following two problems. Firstly, the concept of xeric series is intended to select among the numerous solutions of a differential equation some ``distinguished'' and hence unique ones. However, the choice of this basis is not as ``natural" as one could hope for, a testimony of which is the asymmetric formula in Proposition \ref{prop:trigrel}. Also the exponential function $\exppt$, characterized by Proposition~\ref{prop:period}, and the solutions $\eve(\expp)$ and $\odd(\expp)$ of the differential equation $y+y''=0$ have noticeable properties among the solutions of their respective equations. So the quest for truly distinguished and natural solutions remains open. 
 
Secondly, Problem~\ref{quest:algebraicity}, or, equivalently, Problem~\ref{quest:algmonom}, has been solved for the second order differential equations $y''=y$ and $y''=-y$ in this last section, proving the desired algebraicity. For arbitrary linear differential equations of order greater than or equal to two it is still unclear if the projections of suitably chosen solutions are again algebraic. 
 
\sloppy 
\printbibliography 
\addcontentsline{toc}{section}{References}

\textsc{University of Vienna, Faculty of Mathematics, Oskar-Morgenstern-Platz 1, 1090, Vienna, Austria}\\ 
\textit{Email: }\href{mailto:florian.fuernsinn@univie.ac.at}{\texttt{florian.fuernsinn@univie.ac.at}}\\ 
\textit{Email: }\href{mailto:herwig.hauser@univie.ac.at}{\texttt{herwig.hauser@univie.ac.at}}\\ 
 
\textsc{Chubu University, College of Science and Engineering, Matsumoto-cho, Kasugai-shi, Aichi 487-8501, Japan} 
 
\textsc{Kyoto University, Research Institute for Mathematical Sciences,  Kitashirakawa-Oiwakecho, Sakyo-ku, Kyoto 606-8502, Japan}\\ 
\textit{Email: }\href{mailto:kawanoue@kurims.kyoto-u.ac.jp}{\texttt{kawanoue@kurims.kyoto-u.ac.jp}} 
\end{document}